\documentclass[12pt,a4paper]{article}
\usepackage{amssymb, amsfonts, amsthm, amsmath, amsbsy, bbm, fullpage, mathrsfs, mathtools}

\newtheorem{theorem}{Theorem}[section]
\newtheorem{proposition}[theorem]{Proposition}

\newtheorem{lemma}[theorem]{Lemma}

\newtheorem{remark}[theorem]{Remark}

\usepackage[hidelinks]{hyperref}

\title{Collision properties of the\\four-dimensional random walk trace}
\author{D.~A.~Croydon\footnote{Research Institute for Mathematical Sciences, Kyoto University, croydon@kurims.kyoto-u.ac.jp.}, D.~Shiraishi\footnote{Graduate School of Informatics, Kyoto University, shiraishi@acs.i.kyoto-u.ac.jp.}, S.~Watanabe\footnote{Research Institute for Mathematical Sciences, Kyoto University, swatanab@kurims.kyoto-u.ac.jp.}}
\date{}

\begin{document}

\maketitle

\begin{abstract}
We consider collisions of multiple random walks on the trace of a simple random walk on the four-dimensional integer lattice. For two independent walks (in continuous time), we apply a result of Noda to derive a scaling limit for the collision time process. For three independent walks (in discrete time), we demonstrate infinitely many triple collisions occur.
\end{abstract}

\tableofcontents

\section{Introduction}

In recent years, collision properties of random walks on graphs have been intensively studied. For a given locally-finite infinite graph, a fundamental question is whether two random walks meet infinitely often or not. Whilst a convenient necessary and sufficient condition for the infinite collision property is yet to be established, the property is easily seen to be equivalent to recurrence for transitive graphs \cite{KrishPeres}, and convenient sufficient conditions for it are provided by \cite{BPS} and \cite{HP}. One can also ask the analogous question for three random walks. In this direction, \cite{CDA} demonstrated a phase transition for certain comb-like graphs from a regime where infinite triple collisions occur to a regime where they do not. (Triple collisions on certain graphs were also discussed in the earlier works of \cite{BPS, Chenchen}.) In particular, it was shown in \cite{CDA} that the region in which infinite triple collisions occur is quite narrow, consisting only of combs whose volume growth is at a rate of an order no greater than $n\log n$. From the result of \cite{CDA}, one is led to consider what other graphs might exhibit the infinite triple collision property, and this is the basis of this article. Indeed, we study the trace of the four-dimensional simple random walk, which is a natural random graph, whose volume growth with respect to the resistance metric is of order $n(\log n)^{1/2}$ \cite{CS}. We show that, for this random graph, the infinite triple collision property occurs. The difficulty in obtaining this conclusion lies in the fact that the model sits at the critical dimension for random walk, which means that we have only very weak control on the existence of atypical events, and so we are required to develop a novel approach to understanding the behaviour of the associated random walk. We believe that some of the ideas will be useful for studying aspects of random walks on other models at their critical dimensions. As a simpler problem, we also apply a recent result of \cite{Noda} to obtain a scaling limit for double collisions. (See \cite{CroyRW, exactcorr, CS, DS', exact, DStams} for other works concerning the behaviour of the random walk on the random walk trace.)

In order to state our main results, let us now introduce precisely the main objects of discussion. 
We write $\mathbb{Z}^d$ for the $d$-dimensional  integer lattice; we will only consider $d=4$ in this article. Let $S=(S_k)_{k\geq 0}$ be a simple random walk on $\mathbb{Z}^{4}$, which we will assume satisfies $S_0=0$ (unless otherwise stated); the associated probability measure and expectation will be denoted by $P$ and $E$, respectively. For a given realisation of $S$, we set $\mathcal{G}$ to be the locally-finite infinite graph whose vertex and edge sets are given by
\[V(\mathcal{G}):=\left\{S_k:\:k\geq 0\right\};\qquad E(\mathcal{G}):=\left\{\{S_k,S_{k+1}\}:\:k\geq 0\right\},\]
respectively. Moreover, we suppose $X^i=(X^i_n)_{n\geq 0}$, $i=1,2,3$, are three independent copies of the discrete-time simple random walk on $\mathcal{G}$. We will write $P^\mathcal{G}_{x_1,x_2,x_3}$ for the law of the triple $(X^1,X^2,X^3)$, when started from $(x_1,x_2,x_3)\in V(\mathcal{G})^3$. We say that a triple collision occurs at time $n$ if 
\[X_n^1=X_n^2=X_n^3,\]
and that the infinite triple collision property holds for $\mathcal{G}$ if the number of such times is infinite, $P^\mathcal{G}_{0,0,0}$-almost-surely. The principal goal of this research is to establish the following.

\begin{theorem}\label{mr1}
$P$-almost-surely, the infinite triple collision property holds for $\mathcal{G}$.
\end{theorem}

\begin{remark}
    From our argument, it is possible to check that the number of collisions up to time $n$ is infinitely often greater than $c(\log n)^{1/2}$, almost-surely (cf.\ Proposition \ref{hbound} below). Similarly to the comments of \cite[Remark 1.4]{CDA} it would be of interest to seek more detailed information on the collision rate, such as determining an upper bound on the number of collisions to time $n$ that holds for all time. (We conjecture that, up to smaller order deviations, our lower bound is sharp.)
\end{remark}

To prove Theorem \ref{mr1}, we essentially follow the strategy of \cite{CDA}, which involves a careful analysis of the heat kernel (transition density) for the random walks $X^i$, $i=1,2,3$. The main additional issue in this article, as compared to \cite{CDA}, is that our graph is random. Indeed, it is possible to check that the deterministic comb graphs of \cite{CDA} satisfy suitable volume and resistance bounds down to very small scales, which allows the heat kernel to be well-estimated down to correspondingly small scales uniformly over appropriate spatial regions. Of course, being a random graph, $\mathcal{G}$ will sometimes exhibit `bad' (atypical) sections, and we need to control these. Since we are working in the critical dimension, the probabilistic order of bad events occurring will typically decay logarithmically, which makes taking union bounds impractical, and obtaining uniform heat kernel control seems difficult (and possibly it does not hold). We avoid this issue by requiring instead heat kernel estimates on a smaller `good' set. This relaxes the demands on the probabilistic estimates we need enough to be able to complete the argument, though does require a novel double application of the Paley-Zygmund inequality (see the comments that appear above the statement of Lemma \ref{bilem}). As noted above, we anticipate that a similar approach based on restricting heat kernel estimates to good sets will be useful in the study of other random walks on random graphs at their critical dimension.

As is now a standard technique for establishing such heat kernel estimates, we define the good set to satisfy some volume and effective resistance bounds and require that $\mathcal{G}$ contains the good set with high probability. We require these bounds to hold uniformly down to a small scale, which we do by improving the estimates provided in \cite{exact} (see also \cite{exactcorr}). In particular, we achieved our aim by considering the event where $\mathcal{G}$ has at most one long-range self-intersection, obtaining a smaller error probability on effective resistance than the one obtained in \cite{exact}, so that we can take a union bound to prove that the bound holds simultaneously at a suitably wide range of scales. 

As our next main result, we explain how the conclusion of \cite{Noda} can be applied in our setting. Since \cite{Noda} is written for continuous-time random walks and it is not straightforward to modify the result to yield the corresponding statement for discrete-time random walks, we now consider continuous-time random walks. Let $\tilde{X}^i=(\tilde{X}^i_t)_{t\geq 0}$, $i=1,2$, be two independent copies of the so-called variable-speed random walk on $\mathcal{G}$ started from 0, and write $\tilde{P}^\mathcal{G}_{0,0}$ for their joint law. In particular, each $\tilde{X}^i$ is the continuous-time Markov chain which has a jump rate of 1 along each edge of the graph. Namely, the jump chain of $\tilde{X}^i$ is equal in distribution to $X^i$, and the holding times are exponential, with parameter given by the current vertex degree. Write $C=(C_t)_{t\geq 0}$ for the collision time process  
\begin{equation}\label{ctdef}
C_t:=\int_0^t\mathbf{1}_{\{\tilde{X}^1_t=\tilde{X}^2_t\}}dt.
\end{equation}
We then have the following; a more detailed description of the limiting process can be found in Section \ref{2coll}. Note that, as per (a minor modification of) the main result of \cite{CS}, the reflected Brownian motions that appear in the conclusion of Theorem \ref{mr2} capture the scaling limit of the pair $(\tilde{X}^1,\tilde{X}^2)$. The space $D([0,\infty),\mathbb{R})$ of the statement is the collection of cadlag functions from $[0,\infty)$ to $\mathbb{R}$, equipped with the Skorohod $J_1$ topology.

\begin{theorem}\label{mr2}
There exists a constant $c\in (0,\infty)$ such that, under the annealed law, 
\[\int\tilde{P}^\mathcal{G}_{0,0}\left(\cdot\right)P(d\mathcal{G}),\]
the process
\[\left(n^{-1/2}\log(n)^{1/4}C_{nt}\right)_{t\geq 0}\] converges in distribution as $n\rightarrow \infty$ in the space $D([0,\infty),\mathbb{R})$ to a non-trivial continuous-time stochastic process $C^{B}=(C^{B}_t)_{t\geq 0}$, which represents the collision time process of two independent one-dimensional Brownian motions on $[0,\infty)$, started at and reflected from 0.
\end{theorem}

\begin{remark}
(a) In fact, the result of \cite{Noda} would also allow us to retain information about the spatial positions of the collisions in a `space-time collision measure' (cf.\ \cite{Nguyen}), but we give a simpler statement for brevity.\\
(b) Whilst we do not pursue it here, a similar (but easier version of the) argument would yield a quenched version of Theorem \ref{mr2} in dimensions $d\geq 5$. In particular, in these dimensions, for each fixed realization of the graph $\mathcal{G}$, then the same scaling limit would apply without the appearance of the $\log$ term. (The scaling limit of the random walk itself was originally proved in \cite{CroyRW}.) Furthermore, by arguing as for Theorem \ref{mr1}, it should also be possible to check that the infinite triple collision property holds in these dimensions.\\
(c) It is also interesting to consider the situation for three dimensions. In this case, we have from \cite[Theorem 1.2.3]{DS'} that, for $P$-a.e.\ realisation of $S$ and every $\varepsilon>0$, there exists a constant $C$ such that:
\[p_n(0,0)\leq C n^{-(10/19-\varepsilon)},\qquad \forall n\geq 1,\]
where $(p_n(x,y))_{x,y\in V(\mathcal{G}),n\geq 0}$ is the heat kernel of the discrete-time simple random walk on $\mathcal{G}$, (cf.\ \eqref{hkdef} below). Since it also holds for any infinite graph that $p_n(x,x)\leq Cn^{-1/2}$ for some universal constant, we thus obtain from Cauchy-Schwarz that
\[\sup_{x\in V(\mathcal{G})}p_{2n}(0,x)\leq \sup_{x\in V(\mathcal{G})}\sqrt {p_{2n}(0,0)p_{2n}(x,x)}\leq C\sqrt{n^{-(10/19-\varepsilon)}n^{-1/2}}=Cn^{-(39/76-\varepsilon/2)}\]
for $P$-a.e.\ realisation of $\mathcal{G}$. Arguing as in \cite[Lemma 2.1]{CDA}, it follows that the expected number of collisions is bounded by
\[\sum_{n=0}^\infty P^\mathcal{G}_{0,0,0}\left(X^1_n=X^2_n=X^3_n\right)\leq 64\sum_{n=0}^\infty\sup_{x\in V(\mathcal{G})}p_{2n}(0,x)^2<\infty.\]
In particular, in the three-dimensional case, triple collisions occur finitely often, $P^\mathcal{G}_{0,0,0}$-a.s.\ for $P$-a.e.\ realisation of $S$. A natural further question is whether the infinite double collision property holds in three dimensions. We conjecture that it does, but we are currently unable to provide a proof, for which it seems a more detailed analysis of the three-dimensional random walk path is required. \\
(d) As per \cite[Remark 2.3]{CDA}, four or more walks collide simultaneously only finitely often on any bounded degree graph. Thus, the problem described in part (c) is the only one concerning the infinite collision property of simple random walk traces on the integer lattice that remains.
\end{remark}

The remainder of the article is organised as follows. In Section \ref{volsec}, we derive various estimates on the volume growth and resistance for the graph $\mathcal{G}$, which are then applied in Section \ref{sec:HKbounds} to yield corresponding heat kernel bounds. Next, in Section \ref{3coll}, we derive our first main result, Theorem \ref{mr1}. Finally, we conclude the article with Section \ref{2coll}, which contains the proof of Theorem \ref{mr2}. 

As for notational conventions, we sometimes use a continuous variable $x$ in a place where a discrete one is required; in such cases, $x$ should be understood to mean $\lfloor x \rfloor$. Throughout the paper, $c$, $C$, $c_{1}$,\dots denote constants whose values may change from line to line, unless specifically noted.

\section{Volume and resistance bounds}\label{volsec}

In this section, we will prepare some volume and effective resistance bounds of $\mathcal{G}$ that are needed to obtain heat kernel estimates in the next section. We begin by setting out some preliminaries in Subsection \ref{volsec-pre}. In Subsections \ref{volsec-1} and \ref{volsec-2}, we obtain the estimates required for the heat kernel upper and lower bounds, respectively. In Subsection \ref{volsec-3} we extend the bounds to conditional probability with an eye toward a Borel-Cantelli argument that we will later apply.

\subsection{Notation and preliminary results}\label{volsec-pre}
We write $\mathbb{Z}^d$ and $\mathbb{R}^d$  for the $d$-dimensional  integer lattice and the $d$-dimensional Euclidean space, respectively. For $x \in \mathbb{R}^{d}$, $|x|$ denotes the Euclidean distance between $x$ and the origin. We will consider only the $d = 4$ case in this article. 

Let $\{ a_{n} \}_{n \ge 1}$ be a sequence with  $a_{n} > 0$ for all $n \ge 1$.
For a sequence $\{ b_{n} \}_{n \ge 1}$, we write $b_{n} = o (a_{n})$ if $\lim_{n \to \infty} \frac{b_{n}}{a_{n}} = 0$. We also write $b_{n} = O (a_{n})$ if there exists a constant $C > 0$ such that $|b_{n}| \le C a_{n}$ for all $n \ge 1$. When we want to make it explicit that the rate of convergence to $0$ or the constant $C>0$ depends on a certain parameter $\alpha$, we write $o_{\alpha}(a_n)$ or $O_{\alpha}(a_n)$. As usual, we use to notation $b_n\sim a_n$ to mean that $\lim_{n\to\infty}\frac{b_n}{a_n}=1$. 

We denote by ${\cal G}_{m, n}$ the (random) graph whose vertex set is given by the sites visited by the walker $S$ between times $m$ and $n$, and whose edge set corresponds to pair of consecutive visited sites. Formally, the vertex and edge sets of ${\cal G}_{m, n}$ are
\[\left\{ S_j \mid m \le j \le n \right\},\qquad\left\{ \{ S_j, S_{j+1}\} \mid m \le j \le n-1\right\},\]
respectively. We also set $\mathcal{G}_n:=\mathcal{G}_{0,n}$.

For a graph $G = (V, E)$, let $R_{G} (\cdot, \cdot )$ denote the effective resistance  on $G$. That is, $R_{G}(\cdot,\cdot)$ is defined as follows.

\begin{itemize}
\item For $f, g \in \mathbb{R}^{V}$, write
$${\cal E} (f, g) = \frac{1}{2} \sum_{x, y \in V, \, \{ x, y \} \in E} (f(x) - f(y)) (g(x) - g(y)) $$
for a quadratic form defined on the set of functions $h$ for which $\mathcal{E}(h,h)<\infty$.

\item For  disjoint subsets $A$ and $B$  of $V$, define 
$$ R_{G} (A, B) = \left( \inf \left\{ {\cal E} (f, f) \ : \ f \in \mathbb{R}^{V}, \, {\cal E} (f, f) < \infty, \, f|_{A} = 1, \, f|_{B} = 0 \right\} \right)^{-1}. $$ We write $R_{G} (x, y) = R_{G} (\{ x \}, \{ y \})$ for $x, y \in V$.
\end{itemize}
Also, we let $d_{G} (\cdot, \cdot )$ denote the  graph distance on $G$. 

We say $0 \le k \le n$ is a cut time up to $n$ if $S[0,k] \cap S[k+1, n ] = \emptyset.$ Here, for $0 \le a \le b < \infty$, we denote by $S[a,b]$ the set $\{ S_k : a \le k \le b \}.$ In this article, we may also make a slight abuse of notation to denote by $S[a,b]$ the subgraph of $\mathcal{G}$ restricted to the above vertex set according to the context. We denote by 
\[    B_n\coloneqq S[0,n],\]
the ball around the origin defined by the number of steps. 
For $x\in S[0,\infty)$, we denote by  
\[    B_R(x,r)\coloneqq \{y\in S[0,\infty)\mathrel{:}R_\mathcal{G}(x,y)\le r\}\]
the resistance ball, and define 
\[    V_R(x,r)=|B_R(x,r)|.\]
We also write 
\[    B(x,r)=\{y\in\mathbb{Z}^4\mathrel{:}|y-x|\le r\}\]
for the Euclidean ball.

We end this subsection by introducing a preliminary result on the effective resistance on $\mathcal{G}$. 
Let 
\[    \Psi(n)\coloneqq \frac{E(R_\mathcal{G}(0,S(n)))}{n}.\]
Then \cite[Theorem 1.1]{DS-SW-4d} shows that there exists a constant $c_{res} \in (0, \infty)$ such that
\begin{equation}\label{cite-dssw-mainthm}
    \Psi(n)\sim c_{\mathrm{res}}(\log n)^{-\frac{1}{2}}.
\end{equation}

\subsection{Volume lower bound}\label{volsec-1}

Now, we define a ``good" set $A_N$ by setting
\begin{equation}\label{andef}
A_N\coloneqq \left\{x\in B_N\mathrel{:} V_R^{(N)}(x,rN\Psi(rN))\ge crN\mbox{ for all }r\in\left(e^{(\log N)^{1/2}}/N,1\right]\right\},
\end{equation}
where $V_R^{(N)}(x,r)=|\{y\mathrel{:} R_{\mathcal{G}_N}(x,y)\le r\}|$ and $c\in(0,1)$ is a constant that will be fixed later. We highlight that $A_N$ only depends on the graph $\mathcal{G}_N$. It also holds that $V_R(x,r)\geq V_R^{(N)}(x,r)$ for any $x,r,N$, and so the volume bounds that hold on $A_N$ also hold for the corresponding balls on the full graph $\mathcal{G}$.

The main proposition of this subsection is the following. 
\begin{proposition}\label{prop-vol-low}
There exists a universal constant $p\in(0,\frac{1}{4})$ such that,  for every fixed $\varepsilon \in (0,1)$,  
\[    \inf_{m\in[N^{1-\varepsilon},N]}P(|A_N\cap B_m|\ge pm)\ge p\]
holds for sufficiently large $N$.
\end{proposition}

The main ingredient needed to establish the above proposition is the following lemma. 
\begin{lemma}\label{lem-vol-low-1}
For any $m=0,1,\dots,N$, 
\begin{equation}\label{vol-low-lem}
	P(S_m\in A_N)\ge 1-O\left((\log N)^{-1/20})\right).
\end{equation}
\end{lemma}
\begin{proof}
For $S_m\eqqcolon x\in B_N$, let 
\[
    J_x\coloneqq [(j_x-1)2^{-k}N,j_x2^{-k}N],
\]
where we set $j_x\in\{1,\cdots,N\}$ to be the smallest number such that 
\[
    x\in S[(j_x-1)2^{-k}N,j_x2^{-k}N].
\]
We will first show that
\begin{equation*}
    P\left(V_R^{(N)}(x,2^{-k}N\Psi(2^{-k}N))\ge c2^{-k}N\right)\ge 1-C(\log (2^{-k}N))^{-\frac{3}{2}},
\end{equation*}
for each $k=0,1,\cdots, w_N$, where $w_N$ is defined by 
\[    w_N\coloneqq \left\lceil\frac{\log N-(\log N)^{1/2}}{\log 2}\right\rceil,\]
and the constant $C>0$ on the right-hand side does not depend on $x$. 
Suppose that the following two events occur: 
\begin{gather*}
    \left\{\max_{l,m\in J_x}R_{\mathcal{G}_{l,m}}(S_l,S_m)\le 2^{-k}N\Psi(2^{-k}N)\right\},\\
    \{|S[(j_x-1)2^{-k}N,j_x2^{-k}N]|\ge c2^{-k}N\},
\end{gather*}
Then, by monotonicity of effective resistance, the first event implies 
\[
    \max_{x,y\in S[(j_x-1)2^{-k}N,j_x2^{-k}N]}R_{\mathcal{G}_N}(x,y)\le 2^{-k}N\Psi(2^{-k}N),
\]
and thus 
\[
    V_R^{(N)}(x,2^{-k}N\Psi(2^{-k}N))\ge |S[(j_x-1)2^{-k}N,j_x2^{-k}N]|\ge c2^{-k}N,
\]
holds. By (2.17) of \cite{exact} and by the proof of \cite[Proposition 2.1.2]{DS'}, there exists a uniform constant $c\in(0,1)$ such that 
\[
    P\left(V_R^{(N)}(x,2^{-k}N\Psi(2^{-k}N))\right)\ge 1-O\left((\log (2^{-k}N))^{-\frac{3}{2}}\right)-O\left((2^{-k}N)^{-1}\right),
\]
and the right-hand side is uniform in $x$ by translation invariance. 

Next, we will take the union bound to prove \eqref{vol-low-lem}. 
Note that on the event 
\[    \bigcap_{k=0}^{w_N}\left\{V_R^{(N)}(x,2^{-k}N\Psi(2^{-k}N))\ge c2^{-k}N\right\},\]
we have $\{x\in A_N\}$ by taking $c$ properly, which follows from monotonicity of the volume. 
Thus, we have 
\begin{align*}
    P(x\not\in A_N)&\le \sum_{k=0}^{w_N}C(\log N-k\log 2)^{-3/2}
    \\&\le C\left(\sum_{k=0}^{w_N-(\log N)^{7/10}}\left((\log N)^{1/2}+\log 2(\log N)^{7/10}\right)^{-3/2}\vphantom{\sum_{k=w_N-(\log N)^{7/10}}^{w_N}\left((\log N)^{1/2}\right)^{-3/2}}\right.
    \\&\qquad \qquad+\left.\vphantom{\sum_{k=0}^{w_N-(\log N)^{7/10}}\left((\log N)^{1/2}+\log 2(\log N)^{7/10}\right)^{-3/2}}\sum_{k=w_N-(\log N)^{7/10}}^{w_N}\left((\log N)^{1/2}\right)^{-3/2}\right)
    \\&\le C\left(\log N\cdot (\log N)^{-21/20}+(\log N)^{7/10}\cdot(\log N)^{-3/4}\right)
    \le C(\log N)^{-1/20},
\end{align*}
from which \eqref{vol-low-lem} follows. 
\end{proof}

Before diving into the proof of the main proposition, we recall a lower bound estimate of the volume of $B_n$, which we used in the previous proof. 
\begin{lemma}\label{lem-vol-Bn}
There exists a universal constant $r\in(0,1)$ such that 
\[    P(|B_n|\ge rn)\ge 1-O(n^{-1}).\]
\end{lemma}
\noindent
See, for example, \cite[Proposition 2.1.2]{DS'} for the proof.
\begin{proof}[Proof of Proposition \ref{prop-vol-low}]
Let $m\in[N^{1-\varepsilon},N]$. 
Note that if we define by
\[
    \mathcal{S}(k;m)=\{n\in[0,m]\mathrel{:}|\{n'\ge 0\mathrel{:}S_{n'}=S_n\}|\le k\},
\]
the subset of $[0,m]$ consisting of steps $n$ for which $S_n$ is visited $k$ or less times by $S$, then    
\[
    \sum_{n=0}^m\mathbf{1}\{S_n\in A_N,~n\in\mathcal{S}(k;m)\}\le k|\{S_n\in A_N\mathrel{:}n\in\mathcal{S}(k;m)\}|\le k|A_N\cap B_m|.
\]
Also note that by the Markov property, 
\[
    E[\{n'\ge n\mathrel{:}S_{n'}=S_n\}]=G(0,0)<\infty,
\]
where $G(x,y)$ denotes the Green's function of $S$. 
By Markov's inequality, we have that for any $n\in[0,m]$, 
\begin{equation*}
    P(n\not\in\mathcal{S}(k;m))\le \frac{G(0,0)}{k}.
\end{equation*}
Thus, we have 
\begin{align*}
    E(|A_N\cap B_m|)&\ge \frac{1}{k}E\left(\sum_{n=0}^m\mathbf{1}\{S_n\in A_N,~n\in\mathcal{S}(k;m)\}\right) \\
        &\ge \frac{m}{k}P\left(S_n\in A_N,~n\in\mathcal{S}(k;m)\right)   \\
        &\ge \frac{m}{k}\left(1-O\left((\log N)^{-1/20}\right)-\frac{G(0,0)}{k}\right).
\end{align*}
Moreover, trivial upper bounds yield
\begin{equation*}
    |B_m|\le m;\quad E(|A_N\cap B_m|^2)\le E(|B_m|^2)\le m^2.
\end{equation*}
Set $k=\lceil4G(0,0)\rceil$ and $p=(4k)^{-1}$. 
Applying the Paley–Zygmund inequality (see \cite[Lemma 4.1]{Kall}, for example), we have 
\begin{align*}
       P\left(|A_N\cap B_m|\ge pm\right)
    &\ge P\left(|A_N\cap B_m|\ge \frac{1}{2}E(|A_N\cap B_m|)\right)\\
        &\ge p^2
\end{align*}
for all sufficiently large $N$, which concludes the proof. 
\end{proof}

\subsection{Effective resistance bound}\label{volsec-2}

In this subsection, we prove that a certain good event, upon which we will later show we have a desired heat kernel lower bound, holds with high probability. Let $\varepsilon\in(0,1)$, which will be fixed later. For given constants $c_i$, $i=1,\dots,4$, and $C$, we define $\mathcal{E}_N:=\mathcal{E}_N(C,c_1,c_2,c_3,c_4)$ to be the set of realizations of $S$ such that: for all $n\in [N^{1-\varepsilon},N]$,
\begin{align}
    \sup_{x\in B_n}R(0,x)\leq c_1n\Psi(n),&\qquad R(0,B_{Cn}^c)\leq c_1Cn\Psi(n),    \label{e1}\\
  R\left(0,B_{Cn}^c\right)&\geq c_2Cn\Psi(n),    \label{e2}\\
    |B_n|&\geq c_3 n, \label{e3}\\
    |B_{Cn}|&\leq c_4 Cn.    \label{e4}
\end{align}
(Here, and similarly elsewhere, $B_{Cn^c}$ is the complement of $B_{Cn}$.) We further define $\tilde{\mathcal{E}}_N$ to be the event that $\mathcal{E}_N$ holds and also, for a given constant $c_5$, for all $n\in [N^{1-\varepsilon},N]$,
\begin{equation}\label{e5}
  \sup_{x\in B_{n/c_5C}}R(0,x)\leq \frac{c_1}{c_5C}n\Psi(n).
\end{equation}

We begin with the estimates on a single scale. 
\begin{lemma}\label{lem-eff-e2}
Fix $c_2\in(0,\frac{1}{20})$. 
There exists a universal constant $c_0\ge 1$ such that 
\begin{equation}\label{eff-lem-e2}
    P(R\left(0,B_n^c\right)\geq c_2n\Psi(n))\ge 1-c_0(\log n)^{-1}.
\end{equation}
\end{lemma}
\begin{proof}
For $i\ge 0$, let $t_i=ic_2n$ and 
\[    I_i=[t_{i-1},t_i]\quad \mbox{for }\ i\ge 1.\]
For two intervals $I$ and $I'$, we write 
\begin{equation*}
    I\not\leftrightarrow I' \quad \mbox{if }\ S(I)\cap S(I')=\emptyset.
\end{equation*}
We define the event $F$ by 
\[    F=\left\{I_i\not\leftrightarrow I_j\mbox{ for all }i,j\in\{1,2,\cdots,16\}\mbox{ with }|i-j|>1,\quad [0,t_{15}]\not\leftrightarrow[t_{16},\infty)\right\}.\]
Then by \cite[Theorem 4.3.6]{Lawb}, there exists a universal constant $C_0>0$ such that 
\begin{equation*}
    P(F^c)\le \frac{C_0}{2}(\log n)^{-1},
\end{equation*}
and we can assume $C_0\ge 1$ without loss of generality.

Next, let $a_n=n(\log n)^{-6}$. 
We define the event $G_i~(i=1,2,3,4)$ by 
\begin{equation*}
    G_i=\left\{\begin{gathered}
        \exists k_i\in[t_{4i-2}-a_n,t_{4i-2}],\ \exists k'_i\in[t_{4i-2},t_{4i-2}+a_n]
        \\ \mbox{ such that }[t_{4(i-1)},k_i]\not\leftrightarrow[k_i+1,t_{4i}],\mbox{ and }[t_{4(i-1)},k'_i]\not\leftrightarrow[k'_i+1,t_{4i}]
    \end{gathered}\right\},
\end{equation*}
for $i=1,2,3$ and 
\begin{equation*}
    G_4=\left\{\begin{gathered}
        \exists k_4\in[t_{14}-a_n,t_{14}],\ \exists k'_4\in[t_{14},t_{14}+a_n]
        \\ \mbox{ such that }[t_{14},k_4]\not\leftrightarrow[k_4+1,\infty),\mbox{ and } [t_{14},k'_4]\not\leftrightarrow[k'_4+1,\infty)
    \end{gathered}\right\}.
\end{equation*}
Note that by definition, $G_i$ are independent and 
\[    [0,k_i]\not\leftrightarrow [k_i+i,\infty);\quad [0,k'_i]\not\leftrightarrow [k'_i+1,\infty)\]
hold on $F\cap G_i$. 

By \cite[Lemma 7.7.4]{Lawb}, there exists a universal constant $C'>0$ such that 
\[    P(G_i^c)\le C'\frac{\log \log n}{\log n}.\]
Thus, if we define 
\[    G=\left\{G_i\mbox{ occurs for at least three }i\in\{1,2,3,4\}\right\},\]
then we have 
\[    P(G)\le C''\left(\frac{\log \log n}{\log n}\right)^2,\]
for some universal constant $C''>0$. 

Now, suppose $F\cap G$ occurs. 
We can assume $G_1\cap G_2\cap G_3$ holds since other cases can be handled in the same way. 
Since $k_i\in[t_{4i-2}-a_n,t_{4i-2}],k'_i\in[t_{4i-2},t_{4i-2}+a_n]~(i=1,2,3)$ are cut times, we have 
\begin{align*}
    R&(0,B_n^c) \\
    &=R_{\mathcal{G}_{0,k_1}}(0,S(k_1))+\sum_{i=1}^3 R_{\mathcal{G}_{k_i,k'_i}}(S(k_i),S(k'_i))+\sum_{i=1}^2 R_{\mathcal{G}_{k'_i,k_{i+1}}}(S(k'_i),S(k_{i+1}))\\
    &\qquad +R_{\mathcal{G}_{k'_3,\infty}}(S(k'_3),B_n^c)  \\
    &\ge R_{\mathcal{G}_{0,t_2}}(0,S(t_2))-a_n+\sum_{i=1}^3 R_{\mathcal{G}_{k_i,k'_i}}(S(k_i),S(k'_i))   \\
    &\quad +\sum_{i=1}^2 \left(R_{\mathcal{G}_{t_{4i-2},t_{4i+2}}}(S(t_{4i-2}),S(t_{4i+2}))-2a_n\right)+R_{\mathcal{G}_{k'_3},\infty}(S(k'_3),B_n^c)
\end{align*}
by the triangle inequality. Thus, the event 
\[    \{R(0,B_n^c)\le c_2 n\Psi(n)\}\cap F\cap G_1\cap G_2\cap G_3,\]
implies that the events 
\begin{itemize}
\setlength{\itemsep}{0pt}
    \item $\left\{R_{\mathcal{G}_{0,t_2}}(0,S(t_2))-a_n\le c_2n\Psi(n)\right\}$,
    \item $\left\{R_{\mathcal{G}_{t_{2},t_{6}}}(S(t_{2}),S(t_{6}))-2a_n\le c_2n\Psi(n)\right\}$,\quad and 
    \item $\left\{R_{\mathcal{G}_{t_{6},t_{10}}}(S(t_{6}),S(t_{10}))-2a_n\le c_2n\Psi(n)\right\}$,
\end{itemize}
which are independent of each other, occur. 
Thus, by \cite[Lemma 2.3.1]{exact}, it holds that 
\[    P(\{R(0,B_n^c)\le c_2 n\Psi(n)\}\cap F\cap G)\le C'''\frac{(\log\log n)^3}{(\log n)^{3/2}}\]
for a universal constant $C'''>0$. 
Consequently, we have 
\begin{align*}
    P(\{R(0,B_n^c)\le c_2 n\Psi(n)\})&\le C'''\frac{(\log\log n)}{(\log n)^{3/2}}+P(F^c)+P(G^c)   \\
    &\le c_0(\log n)^{-1},
\end{align*}
from which \eqref{eff-lem-e2} follows. 
\end{proof}

\begin{lemma}\label{lem-e1-e5}
Fix $c_1>1$, $c_4\ge 1$ and $C\ge 1$. 
There exist a universal constant $c_3\in(0,1)$, and constants $c'=c'(c_1)>0$ and $N_1=N_1(c_1,c_5,C)$ such that 
\begin{align*}
	P\left(\begin{gathered}
	\sup_{x\in B_n}R(0,x)\le c_1n\Psi(n),\ R(0,B_{Cn}^c)\le c_1Cn\Psi(n),
		\\ |B_n|\ge c_3n,\ |B_{Cn}|\le c_4Cn,\ \sup_{x\in B_n/c_5C}R(0,x)\le \frac{c_1}{c_5C}n\Psi(n)
	\end{gathered}\right)\ge 1-c'(\log n)^{-3/2},
\end{align*}
hold for $n\ge N_1$. 
\end{lemma}
\begin{proof}
Let $c_3$ be $r$ taken in Lemma \ref{lem-vol-Bn}. Then we have 
\[	P(|B_n|\ge c_3n,\ |B_{Cn}|\le c_4Cn)\ge 1-O(n^{-1}),\]
for $c_4\ge 1$. 
Let 
\[	\partial B_{Cn}=\{x\in S[0,\infty)\setminus B_{Cn}\mathrel{:}\exists y\in B_n\mbox{ such that }\{x,y\}\mbox{ is an edge of }\mathcal{G}\},\]
\textit{i.e.}\ $\partial B_n$ is the outer boundary of $B_{Cn}$ in $\mathcal{G}$. 
Let $x\in \partial B_{Cn}$ and take a neighboring vertex $y\in B_{Cn}$. Then we have 
\begin{equation*}
	R(0,B_{Cn}^c)\le R(0,x)\le R(0,y)+1\le \sup_{y\in B_{Cn}}R(0,y)+1.
\end{equation*}
Since $\Psi$ is slowly-varying, the statement follows from (2.17) of \cite{exact}. 
\end{proof}

\begin{proposition}\label{prop-tilde-EN}
For any $\delta\in(0,1)$, there exists $\varepsilon\in(0,1)$ such that 
\[    P\left(\tilde{\mathcal{E}}_N\right)\ge 1-\delta\]
holds for sufficiently large $N$. 
\end{proposition}
\begin{proof}
Let $c_0>0$ be as defined in Lemma \ref{lem-eff-e2}. 
For fixed $c_1,c_2,c_4$ and $C$, and $c_3$ taken as in Lemma \ref{lem-vol-Bn}, we define the event $J_n$ by 
\[    J_n=\{\mbox{\eqref{e1}--\eqref{e5} hold for }n\}.\]
Let 
\[    v_N=\left\lceil\frac{\varepsilon\log N}{\log 2}\right\rceil.\]
Note that by monotonicity of volume and effective resistance, it suffices to show that the event $\displaystyle \bigcap_{k=0}^{v_N} J_{2^kN}$ occurs with probability at least $1-\delta$.  

It follows from Lemmas \ref{lem-eff-e2} and \ref{lem-e1-e5} that 
\[    P(J_n^c)\le 2c_0(\log n)^{-1}.\]
By taking the union bound, we have 
\begin{align*}
    P\left(\tilde{\mathcal{E}}_N^c\right)&\le \sum_{k=0}^{v_N}2c_0(\log (2^{-k}N))^{-1}    \\
        &\le \left\lceil\frac{\varepsilon\log N}{\log 2}\right\rceil\frac{2c_0}{(1-\varepsilon)}(\log N)^{-1}\le \frac{2c_0\varepsilon}{\log 2(1-\varepsilon)}.
\end{align*}
Taking $\varepsilon\le 2\delta(c_0+2\delta)^{-1}$, we obtain the desired bound. 
\end{proof}

\subsection{Conditional probability}\label{volsec-3}
Let 
\begin{equation}\label{nidef}
	N_i=\exp\left\{\exp\{\exp\{ai\}\}\right\}
\end{equation}
where  
\[    a=1+\frac{(c_0^{-1}\cdot c_{\mathrm{res}})\vee e}{c_1},\]
with $c_{\mathrm{res}}$ being the constant defined in \cite[Theorem 1.1]{DS-SW-4d} (see \eqref{cite-dssw-mainthm}).
Let 
\begin{equation}\label{Fidef}
	\mathcal{F}_i=\sigma\left(\left\{S(j)\mathrel{:}0\le j\le N_i\right\}\right),
\end{equation}
be the $\sigma$-field generated by $\{S(j)\mathrel{:}0\le j\le N_i\}$. 
The aim of this subsection is to prove the following proposition. 
\begin{proposition}\label{prop-conditional-AE}
Let $\tilde{\mathcal{E}}_N$ be as defined in the previous subsection. 
Fix $c_1>1,c_2\in(0,\frac{1}{20}),c_4\ge 1$ and $C\ge 1$.  
There exists universal constants $p'\in(0,1)$ and $c_3\in(0,1)$ such that for any sufficiently large $i$,
\[\inf_{m\in[N_i^{1-\varepsilon},N_i]}P\left(\{|A_{N_i}\cap B_m|\ge p'm\}\cap \tilde{\mathcal{E}}_{N_i}\mathrel{}\middle|\mathrel{} \mathcal{F}_{i-1}\right)\ge p',\quad\mbox{a.s.}\]
\end{proposition}
\begin{proof}
We begin with the volume lower bound. Note that if $x\in S[N_{i-1},N_i]$ satisfies 
\begin{align*}
	\left|V_R(x,rN_i\Psi(rN_i))\cap S[N_{i-1},N_i]\right|\ge crN_i\quad\mbox{for all }r\in\left(e^{(\log N_i)^{1/2}}/N_i,1\right],
\end{align*}
then $x\in A_{N_i}$. 
Furthermore, by monotonicity of the effective resistance, we have that 
\begin{align*}
	&\left\{\left|V_R(x,rN_i\Psi(rN_i))\cap S[N_{i-1},N_i]\right|\ge crN_i\right\}
	\\ &\quad\supset\left\{
	\left|\{y\in S[N_{i-1},N_i]\mathrel{:}R_{\mathcal{G}_{N_{i-1},\infty}}(x,y)\le rN_i\Psi(rN_i)\}\right|\ge crN_i
	\right\}
\end{align*}
Note that the event on the right-hand side is independent of $S[0,N_{i-1}-1]$. 
Thus, by the strong Markov property and translation invariance, we have that for $x\in S[N_{i-1},N_i]$, 
\[    P(x\in A_{N_i}\mid \mathcal{F}_{i-1})\ge 1-O((\log N_i)^{-\frac{1}{20}}),\quad\mbox{a.s.}\]
Since $N_i-N_{i-1}\ge N_i/2$, almost the same argument as Proposition \ref{prop-vol-low} yields that 
\begin{align*}
    P(|\mathcal{A}_{N_i}\cap B_m|\ge pm\mid\mathcal{F}_i)&\ge P\left(|A_{N_i}\cap S[N_{i-1},N_i]\cap B_m|\ge p'm \mid\mathcal{F}_{i-1}\right)   \\
        &\ge p'(1-(\log N_i)^{-\frac{1}{20}}),
\end{align*}
for some $p'$. 

Next, we will handle $\tilde{\mathcal{E}}_{N_i}$. 
It is trivial that  
\[	|B_{N_i}|\ge |S[N_{i-1},N_i]|,\]
where the right-hand side is independent of $S[0,N_{i-1}-1]$. 
Note that for any $x\in B_{N_{i-1}}$ and $y\in B_n$ with $n\in [N_i^{1-\varepsilon},N_i]$, 
\begin{align*}
	R(0,x)&\le N_{i-1}\le c_1 n\Psi(n),	\\
	R(0,y)&\le R(0,S(N_{i-1}))+R(S(N_{i-1}),y)	\\
		&\le R_{\mathcal{G}_{N_{i-1}},\infty}(S(N_{i-1}),y)+\frac{c_1-1}{2}n\Psi(n),
\end{align*}
by the definition of $N_i$.  
Also note that $R_{\mathcal{G}_{N_{i-1}},\infty}(S(N_{i-1}),y)$ is independent of $S[0,N_{i-1}]$. 
Thus, the same argument as the previous section implies that for $n\in[N_i^{1-\varepsilon},N_i]$, 
\begin{align*}
	P\left(\mbox{\eqref{e1} and \eqref{e3}--\eqref{e5} holds}\mid \mathcal{F}_{i-1}\right)\ge 1-O((\log n)^{-3/2}). 
\end{align*}

Now we handle the conditional probability that \eqref{e2} holds. Define $b_i:=N_{i-1}+N_i^{1-\varepsilon}(\log N_i)^{-6}$, let $x_i=S(b_i)$ and let $\hat{g}$ be the Green kernel of the simple random walk on $\mathcal{G}$ killed at the first exit time of $B_{Cn}$. 
For any $C>0$, 
\begin{align}
	|R(0,B_{Cn}^c)-R(x,B_{Cn}^c)|&=|\hat{g}(0,0)-\hat{g}(x_i,x_i)|	\notag\\
		&\le |\hat{g}(0,0)-\hat{g}(0,x_i)|+|\hat{g}(0,x_i)-\hat{g}(x_i,x_i)|	\notag\\
		&\le 2R_{B_{CN_i}}(0,x_i) \notag\\
        &\le 2R_{\mathcal{G}_{0,b_i}}(0,x_i)\le 2N_i^{1-\varepsilon}(\log N_i)^{-6},\notag
\end{align}
where we applied \cite[Corollary 4.1.6]{Kumsf} to obtain the second inequality. 
Thus, we have 
\begin{equation}\label{prop-cond-01}
    P(R(0,B_{Cn}^c)\le c_2Cn\Psi(n)\mid\mathcal{F}_{i-1})\le P(R(x_i,B^c_{Cn})\le 2c_2Cn\Psi(n)\mid \mathcal{F}_{i-1}).
\end{equation}
Furthermore, by \cite[Theorem 1.2.1]{Lawb}, we have 
\[    P\left(\left|S(N_i^{1-\varepsilon}(\log N_i)^{-6})\right|\ge \frac{(N_i^{1-\varepsilon})^{1/2}}{(\log N_i)^4}\right)\ge 1-O\left((\log N_i)^{-4}\right),\]
and by translation invariance and the strong Markov property, this implies 
\[    P\left(\left.\left|S(N_{i-1})-x_i\right|\ge \frac{(N_i^{1-\varepsilon})^{1/2}}{(\log N_i)^4}\ \right|\ \mathcal{F}_{i-1}\right)\ge 1-O\left((\log N_i)^{-4}\right).\]
Thus, applying \cite[Proposition 1.5.10]{Lawb}, we obtain 
\begin{align*}
    P&(S[0,N_{i-1}]\cap S\left[N_{i-1}+N_i(\log N_i)^{-6},\infty\right)\neq\emptyset\mid\mathcal{F}_{i-1}) \\
        &\le P\left(\begin{gathered}B(S(N_{i-1}),N_{i-1})\cap S[N_{i-1}+N_i(\log N_i)^{-6},\infty)=\emptyset,    \\
        \left|S(N_{i-1})-x_i\right|\ge \frac{(N_i^{1-\varepsilon})^{1/2}}{(\log N_i)^4}\end{gathered}
        \mathrel{}\middle|\mathrel{}\mathcal{F}_{i-1}\right)   \\
        &\quad +P\left(\left|S(N_{i-1})-x_i\right|\le \frac{(N_i^{1-\varepsilon})^{1/2}}{(\log N_i)^4}\mathrel{}\middle|\mathrel{} \mathcal{F}_i\right)  \\
        &\le \max_{x\in \mathbb{Z}^4\setminus B(0,(\log N_i)^{-4}(N_i^{1-\varepsilon})^{1/2})}P(B(0,N_{i-1})\cap S[0,\infty)\neq\emptyset)+O\left((\log N_i)^{-4}\right)   \\
        &\le O\left(N_i^{-1+2\varepsilon}\right)+O\left((\log N_i)^{-4}\right)\le O\left((\log N_i)^{-4}\right),
\end{align*}
since $N_{i-1}\ll N_i^{\frac{\varepsilon}{2}}$ by definition. 
Note that on the event  
\[    \left\{S[0,N_{i-1}]\cap S[N_{i-1}+N_i\left(\log N_i)^{-6}\right),\infty)=\emptyset\right\},\]
it holds that
\[    R(x,B_{Cn}^c)=R_{\mathcal{G}_{N_{i-1},\infty}}(x_i,B_{Cn}^c).\] 
As a consequence, the right-hand side of \eqref{prop-cond-01} is bounded above by 
\begin{align*}
    P(R(0,B_{Cn}^c)\le c_2Cn\Psi(n)\mid\mathcal{F}_{i-1})&\le P(R(x_i,B^c_{Cn})\le 2c_2Cn\Psi(n)\mid \mathcal{F}_{i-1})   \\
        &\le P(R_{\mathcal{G}_{N_{i-1},\infty}}(x_i,B_{Cn}^c)\mid\mathcal{F}_{i-1})+O\left((\log N_i)^{-4}\right)   \\
        &\le O((\log n)^{-1})
\end{align*}
where the last inequality follows from the translation invariance. 

Finally, repeating the same argument as Proposition \ref{prop-tilde-EN} and taking $\delta=p'/2$, we obtain the desired result. 

\end{proof}

\section{Heat kernel bounds}\label{sec:HKbounds}

The aim of this section is to derive heat kernel bounds that hold on certain `good' events. The main quantity of interest is defined by setting, for a given realisation of $\mathcal{G}$, 
\begin{equation}\label{hkdef}
    p_n(x,y):=\frac{P^\mathcal{G}_x\left(X_n=y\right)}{\mathrm{deg}(y)},\qquad x,y\in V(\mathcal{G}),\:n\geq 0
\end{equation}
where $X=(X_n)_{n\geq 0}$ is the discrete-time simple random walk on $\mathcal{G}$, $P^\mathcal{G}_x$ is its law when starting from $x$, and $\mathrm{deg}(y)$ is the usual graph degree of $y$ in $\mathcal{G}$. We start by presenting some heat kernel upper bounds as Proposition \ref{hkupper}.  In particular, we give upper heat kernel bounds for points in the set $A_N$, which was defined at \eqref{andef}, using only the lower volume bound in the argument. We also give appropriate bounds for points close to the origin when this point is in $A_N$. As for the lower heat kernel bound, given as Proposition \ref{hklower} below, we again follow a widely-applied approach, but, in addition to volume information, we need to use some more detailed information about the resistance on the space.

\begin{proposition}\label{hkupper}
For some deterministic constants $c_1$, $c_2$ and $C$, the following holds $P$-a.s. If $x\in A_N$ for some $N\geq 1$, then
\[p_n({x},{x})\leq \left\{\begin{array}{ll}
  C\sqrt{1/n},   &  \text{ for every }n\geq 1,\\
  C\sqrt{\Psi(n)/n},   & \text{ for every }n\in [c_1\exp(3\log^{1/2}(N)),c_2N^2\Psi(N)].
\end{array}\right.\]
Moreover, if $0\in A_N$, then the same bounds hold for every $x\in B_{c_0\sqrt{n/\Psi(n)}}$, i.e.\
\[\sup_{x\in B_{c_0\sqrt{n/\Psi(n)}}} p_n(x,x)\leq \left\{\begin{array}{ll}
  C\sqrt{1/n},   &  \text{ for every }n\geq 1,\\
  C\sqrt{\Psi(n)/n},   & \text{ for every }n\in [c_1\exp(3\log^{1/2}(N)),c_2N^2\Psi(N)],
\end{array}\right.\]
where again $c_0$ is a deterministic constant.
\end{proposition}
\begin{proof} We start by checking part the claim for a general $x\in A_N$. By a simple parameterization (setting $m=rn\Psi(rn)$), we have from the definition of $A_N$ that, for $x\in A_N$,
\begin{equation}\label{vb1}
V_R\left(x,m\right)\geq c m\log(m)^{1/2},
\end{equation}
whenever
\begin{equation}\label{mrange}
\exp(\log^{1/2}(N))\Psi(\exp(\log^{1/2}(N)))\leq m\leq N\Psi(N).
\end{equation}
Moreover, since the resistance distance is bounded above by the graph distance, and the graph we are considering is unbounded, it moreover holds that
\begin{equation}\label{vb2}
V_R\left(x,m\right)\geq c m,
\end{equation}
for any $m\geq 1$. 

Now, it holds for any $n,m\geq 1$ that
\begin{equation}\label{hk1}
p_{2n}(x,x)\leq \frac{4m}{n}+\frac{4}{V_R(x,m)},
\end{equation}
see the proof of \cite[Lemma 11]{CHllt}, for example. (The latter result is a simple modification of \cite[Proposition 3.2]{BCK}.) Thus, applying \eqref{vb2} with $m=n^{1/2}$ yields
\[p_{2n}(x,x)\leq \frac{C}{n^{1/2}},\]
for any $n\geq 1$, which gives the first claim of the result. As for the second claim, we set $m=\sqrt{n\Psi(n)}$. When $m$ is in the range at \eqref{mrange}, the volume bound at \eqref{vb1} can then be applied in conjunction with \eqref{hk1} to give
\[p_{2n}(x,x)\leq C\sqrt{\frac{\Psi(n)}{n}},\]
as desired. It remains to translate the range of $m$ to the range of $n$. In particular, if $n\geq c_1\exp(3\log^{1/2}(N))$, then 
\begin{eqnarray*}
m&\geq& \sqrt{c_1\exp(3\log^{1/2}(N))\Psi(c_1\exp(3\log^{1/2}(N)))}\\
&\geq & c(c_1)\exp\left(\tfrac{3}{2}\log^{1/2}(N)\right)\log(N)^{-1/4}\\
&\geq &c(c_1)\exp\left(\log^{1/2}(N)\right),
\end{eqnarray*}
where the constant $c(c_1)$ can be made greater than 1 by taking $c_1$ suitably large. Similarly, if $n\leq c_2N^2\Psi(N)$, then
\[m\leq \sqrt{c_2N^2\Psi(N)\Psi(c_2N^2\Psi(N))}\leq C(c_2)N\Psi(N),\]
where the constant $C(c_2)$ can be made smaller than 1 by taking $c_2$ suitably small. Hence, for such choices of $c_1$ and $c_2$, if $n\in [c_1\exp(3\log^{1/2}(N)),c_2N^2\Psi(N)]$, then $m$ is indeed in the range at \eqref{mrange}. This completes the first part of the proof.

The remaining claim is proved in a similar fashion. If $x\in B_{m/2}$, then $R(0,x)\leq m/2$, and so, if and $0\in A_N$, we obtain that
\[V_R\left(x,m\right)\geq V_R(0,m/2)\geq c (m/2)\log(m/2)^{1/2},\]
whenever $m/2$ is in the range described at \eqref{mrange}. From this, the heat kernel bound follows. Note that, by suitably adjusting the constants, we can assume they are the same in both parts of the lemma. 
\end{proof}

In order to state our heat kernel lower bound, recall the good event $\tilde{\mathcal{E}}_N$ defined in Section \ref{volsec-2} (see \eqref{e1}--\eqref{e5}); this is the event upon which our argument will apply. Moreover, for a subset $B\subseteq V(\mathcal{G})$, we write 
\[p^B_n(x,y):=\frac{P^\mathcal{G}_x\left(X_n=y,\:\tau_B^c>n\right)}{\mathrm{deg}(y)},\qquad x,y\in V(\mathcal{G}),\:n\geq 0,\]
for the heat kernel when $X$ is killed at 
\[\tau_B^c:=\inf\left\{n\geq 0:\:X_n\not\in B\right\},\]
that is, the exit time of $X$ from $B$. We note that the proof of the following result follows that of \cite[Lemma 3.2]{CDA}. For the statement of the result, we define $d(x,y)$ to be the usual graph distance between vertices $x,y\in V(\mathcal{G})$.

\begin{proposition}\label{hklower} 
If $\tilde{\mathcal{E}}_N$ holds with $C\geq 1\vee (12c_1/c_2)$ and $c_5\geq 32c_4/c_3$, then there exists constants $\tilde{c}_1$, $\tilde{c}_2$, $\tilde{c}_3$ such that
\[\inf_{y\in B_{\tilde{c}_0\sqrt{m\Psi(m)}}}p_{m}^{B_{CN}}(y,x)\geq \tilde{c}_1\sqrt\frac{\Psi(m)}{m}\mathbf{1}_{\{m-d(y,x)\text{ is even}\}},\]
for all $x\in B_{\tilde{c}_2\sqrt{m/\Psi(m)}}$, $m\in [N^{2-\varepsilon},\tilde{c}_3N^2\Psi(N)]$.
\end{proposition}
\begin{proof}
We start by showing a version of the result for the transition density of the random walk started from the origin. In particular, we will check that: if $\tilde{\mathcal{E}}_N$ holds with $C\geq 1\vee (3c_1/c_2)$ and $c_5\geq 32c_4/c_3$, then there exists constants $\tilde{c}_1$, $\tilde{c}_2$, $\tilde{c}_3$ such that
\begin{equation}\label{0bound}
p_{m}^{B_{CN}}(0,x)\geq \tilde{c}_1\sqrt\frac{\Psi(m)}{m}\mathbf{1}_{\{m-d(0,x)\text{ is even}\}},\quad \forall x\in B_{\tilde{c}_2\sqrt{m/\Psi(m)}},\quad m\in [N^{2-\varepsilon},\tilde{c}_3N^2\Psi(N)].
\end{equation}

For the first part of the proof, we only need to assume that the event $\mathcal{E}_N$ holds with $Cc_2\geq 3c_1$. To begin with, we simply observe that
\begin{equation}\label{exit}
    E^\mathcal{G}_0\tau_{B_{Cn}^c}=\sum_{x\in B_{Cn}}g_{B_{Cn}}(x)\mu(\{x\}),
\end{equation}
where $\mu$ is the measure on $V(\mathcal{G})$ satisfying $\mu(\{x\})=\mathrm{deg}(x)$, and $g_{B_{Cn}}(x)$ is the expected occupation density of $X$ at $x$ when started from 0 and run up $\tau_{B_{Cn}^c}$; that is,
\[g_{B_{Cn}}(x):=\frac{1}{\mu(\{x\})}E^\mathcal{G}_0\sum_{m=0}^{\tau_{B_{Cn}^c}-1}\mathbf{1}_{\{X_m=x\}}.\]
In particular, we have from \cite[Proposition 4.1.4]{Kumsf} that
\[2g_{B_{Cn}}(x) = R(0, B_{Cn}^c) +R(x, B_{Cn}^c)-R_{ B_{Cn}^c}(0,x),\]
where $R_{ B_{Cn}^c}(0,x)$ represents the effective resistance between 0 and $x$ when $B_{Cn}^c$ is `fused', see \cite[Section 2]{Kumsf} for background on effective resistance and \cite[(4.9)]{Kumsf} specifically for the definition of the fused effective resistance. Since the effective resistance is a metric, we can apply the triangle inequality to deduce that
\[g_{B_{Cn}}(x)\geq R(0, B_{Cn}^c) - R(0,x).\]
(Here, we have also used the fact that $R_{ B_{Cn}^c}(0,x)\leq R(0,x)$.) In particular, applying \eqref{e1} and \eqref{e2}, we observe that, for $x\in B_n$ (with $n$ in the appropriate range), 
\[g_{B_{Cn}}(x)\geq (c_2C-c_1)n\Psi(n).\]
Since we are assuming that $c_2C\geq 3c_1$, if we insert this into \eqref{exit}, we are able to deduce that
\begin{equation}\label{lb}
E^\mathcal{G}_0\tau_{B_{Cn}^c}\geq c_1n\Psi(n)\mu(B_n)\geq c_1c_3n^2\Psi(n),
\end{equation}
where the second inequality follows from \eqref{e3}.

Next, applying the Markov property at time $m$, it is also possible to deduce that
\begin{equation}\label{mp}
E^\mathcal{G}_0\tau_{B_{Cn}^c}\leq m+ P^\mathcal{G}_0\left(\tau_{B_{Cn}^c}>m\right)\sup_{x\in B_{Cn}}E^\mathcal{G}_x\tau_{B_{Cn}^c}.
\end{equation}
By the commute time identity (see \cite[Proposition 10.6]{LPW}, for example), the supremum on the right-hand side can be bounded as follows:
\[\sup_{x\in B_{Cn}}E^\mathcal{G}_x\tau_{B_{Cn}^c}\leq \sup_{x\in B_{Cn}}R(x,B_{Cn}^c)\mu(B_{Cn})\leq 2c_1c_4C^2n^2\Psi(n),\]
where we have also applied \eqref{e1} and \eqref{e4}. Returning to \eqref{mp}, from this and \eqref{lb}, we obtain
\[P^\mathcal{G}_0\left(\tau_{B_{Cn}^c}>m\right)\geq \frac{ c_1c_3n^2\Psi(n)-m}{ 2c_1c_4C^2n^2\Psi(n)}.\]
Taking $m=\tfrac{1}{2}c_1c_3n^2\Psi(n)$, this implies
\begin{equation}\label{constlower}
    P^\mathcal{G}_0\left(\tau_{B_{Cn}^c}>\tfrac{1}{2}c_1c_3n^2\Psi(n)\right)\geq \frac{ c_3}{ 4c_4C^2}.
\end{equation}
From this, to give a lower estimate on the on-diagonal part of the heat kernel, we follow  \cite[Proposition 4.3.4]{Kumsf}, for example, to deduce that, for any $m\geq 1$,
\[P^\mathcal{G}_0\left(\tau_{B_{Cn}^c}>m\right)^2= \left(\sum_{x\in B_{Cn}}p_m^{B_{Cn}}(0,x)\mu(\{x\})\right)^2\leq  p_{2m}^{B_{Cn}}(0,0)\mu(B_{Cn})\leq 4c_4Cn,\]
where the second inequality here is an application of the Cauchy-Schwarz inequality, and the third is a consequence of \eqref{e4}. In particular, taking $m$ to be equal to $m(n):=\lfloor \tfrac{1}{2}c_1c_3n^2\Psi(n)\rfloor$, combining this bound with \eqref{constlower} yields
\begin{equation}\label{odlower}
   p_{2m(n)}^{B_{Cn}}(0,0)\geq \frac{c_3}{8c_4^2C^3n}\geq \frac{c_6}{C^3}\sqrt\frac{\Psi(m(n))}{m(n)},
\end{equation}
where the constant $c_6$ depends only upon $c_1$, $c_3$ and $c_4$.

The next step will be to extend the bound at \eqref{odlower} to hold for other points. We will now also need to incorporate the bound at \eqref{e5}. By the symmetry of the transition density: for any $m\geq 1$ such that $m-d(0,x)$ is even,
\begin{align}
\nonumber p_{m}^{B_{Cn}}(0,x)&=p_{m}^{B_{Cn}}(x,0)\\
&=\sum_{k=0}^{\lfloor m/2\rfloor}P^\mathcal{G}_x\left(\tau_{B_{Cn}^c}>m-2k,\:\tau_{0}=m-2k\right)p^{B_{Cn}}_{2k}(0,0)\nonumber\\
&\geq P^\mathcal{G}_x\left(\tau_{0}\leq m\wedge \tau_{B_{Cn}^c}\right)p^{B_{Cn}}_{2\lfloor m/2\rfloor}(0,0),\label{ilb}
\end{align}
(cf.\ \cite[Lemma 3.4]{Chenchen}), where we have used the notation $a\wedge b:=\min\{a,b\}$, and the second equality is obtained by conditioning on the time $X$ first hits $0$ when started from $x$, and the inequality holds because $p^{B_{Cn}}_{2k}(0,0)$ is decreasing in $k$ (see \cite[(3.3)]{BPS}, for example). In particular, given a suitable bound on the probability in the final expression, this estimate will allow us to apply the estimate at \eqref{odlower} to obtain the desired result. Observe that
\[P^\mathcal{G}_x\left(\tau_{0}> m\wedge \tau_{B_{Cn}^c}\right)\leq P^\mathcal{G}_x\left(\tau_{0}\wedge \tau_{B_{Cn}^c}>m\right)+P^\mathcal{G}_x\left(\tau_{0}> \tau_{B_{Cn}^c}\right).\]
Appealing again to the commute time identity, the first term is bounded as follows:
\[P^\mathcal{G}_x\left(\tau_{0}\wedge \tau_{B_{Cn}^c}>m\right)\leq \frac{R(0,x)\mu(B_{Cn})}{m}\leq \frac{4c_1c_4n^2\Psi(n)}{c_5m},\]
where for the second inequality we have applied \eqref{e4} and \eqref{e5}. Additionally, by \cite[Exercise 2.36]{LP}, for example, together with \eqref{e1} and \eqref{e2}, the second term satisfies
\[P^\mathcal{G}_x\left(\tau_{0}> \tau_{B_{Cn}^c}\right)\leq \frac{R(0,x)}{R({x},B_{Cn}^c)}\leq\frac{c_1}{c_2C-c_1}.\]
Hence, 
\[P^\mathcal{G}_x\left(\tau_{0}> m\wedge \tau_{B_{Cn}^c}\right)\leq \frac{4c_1c_4n^2\Psi(n)}{c_5m}+\frac{c_1}{c_2C}.\]
Since we are assuming that $c_2C\geq 3c_1$, the second term here is bounded above by $\tfrac{1}{2}$. Moreover, we now set $m$ to be equal to $m'(n)$, which is defined to be the unique integer such that $m'(n)-d(0,x)$ is even and $\lfloor m'(n)/2\rfloor =m(n)$, where $m(n)$ was defined earlier in the proof. Then, the first term is bounded above by
\[ \frac{4c_1c_4n^2\Psi(n)}{c_5m'(n)}\leq\frac{2c_1c_4n^2\Psi(n)}{c_5m(n)}\leq \frac14 ,\]
where to ensure the final inequality we apply the assumption that $c_5\geq 32c_4/c_3$; this might require us to also take $N$ to be suitably large. Returning to \eqref{ilb}, we thus conclude, for this choice of $m'(n)$,
\[p_{m'(n)}^{B_{CN}}(0,x)\geq p_{m'(n)}^{B_{Cn}}(0,x)\geq \frac{1}{4}p^{B_{Cn}}_{2m(n)}(0,0)\geq  \frac{c_6}4\sqrt\frac{\Psi(m(n))}{m(n)},\qquad \forall x\in B_{n/c_5C},\]
where the first inequality is obvious, since the process clearly leaves $B_{Cn}$ before it leaves $B_{CN}$, and, for the final inequality, we recall \eqref{odlower}. 

We can conclude the proof of \eqref{0bound} by a simple reparameterisation of the previous bound, writing the various expressions in terms of $m$. (To ensure $m'(n)\geq N^{2-\varepsilon}$, we should take $N$ sufficiently large, but the bound is readily extended to all $N$ by reducing the constant that appears in the place of $\tilde{c}_2$ if necessary.)

To extend \eqref{0bound} to more general starting points, we first note that if $\tilde{\mathcal{E}}_N$ holds with $C\geq 1$, and $y\in B_{\delta n\Psi (n)}$ for some $n\in [N^{1-\varepsilon},N]$ then:
\begin{itemize}
    \item \eqref{e1} holds with $0$ replaced by $y$ and $c_1$ replaced by $c_1+\delta$;
    \item \eqref{e2} holds with $0$ replaced by $y$ and $c_2$ replaced by $c_2-\delta$;
    \item \eqref{e5} holds with $0$ replaced by $y$ and $c_1/c_5C$ replaced by $c_1/c_5C+\delta$.
\end{itemize}
Hence, if $\delta$ is taken to be equal to $\min\{c_1,\frac{c_2}{2},\frac{c_1}{c_5C}\}$, then the various estimates of the event $\tilde{\mathcal{E}}_N(C,2c_1,c_2/2,c_3,c_4)$ hold with 0 replaced by $y\in B_{\delta n\Psi (n)}$ and $n\in [N^{1-\varepsilon},N]$. As a consequence, the desired result follows on replacing 0 by $y$ in the proof of \eqref{0bound}. Note that the range of $y$ in the statement of the proposition comes from the fact that
$\delta n \Psi(n)$ is of the order of $c\sqrt{m'(n)\Psi(m'(n)}$.
\end{proof}

\section{Collisions of three independent random walks}\label{3coll}

In order to establish Theorem \ref{mr1}, i.e.\ show that infinitely many triple collisions occur almost-surely, for almost-every realization of $\mathcal{G}$, we  introduce the counting random variable
\[H_i\coloneqq \sum_{n=\Theta_{\eta N_{i-1}}+1}^{(\Theta_{\eta N_{i-1}}+c_2N_i^2\Psi(N_i))\wedge \Theta_{\eta N_{i}}}\mathbbm{1}_{\{X^1_n=X^2_n=X^3_n\}},\]
where $c_2$ is the constant of Proposition \ref{hkupper} (which we can assume is no greater than 1), $\eta\in(0,1)$ is another constant, $N_i$ are the constants defined at \eqref{nidef}, and
\[\Theta_{N}:=\inf\left\{n\geq 0:\:\{X^1_n,X^2_n,X^3_n\}\not\subseteq B_{N}\right\}\]
denotes the first time at which one of the random walkers exits $B_N$. In particular, $H_i$ counts the number of triple collisions occurring at level $i$. The main goal of this section is to prove the following result, from which Theorem \ref{mr1} obviously follows.

\begin{proposition}\label{hbound}
There exists a deterministic constant $c$ such that, for $P$-almost-every realization of $\mathcal{G}$, it $P^\mathcal{G}_{0,0,0}$-almost-surely holds that
    \[H_i\geq c\left(\log N_i\right)^{1/2}\]
    for infinitely many $i$.
\end{proposition}

To prove Proposition \ref{hbound}, we will apply a Borel-Cantelli argument, and for this, we need to suitably decouple the events on different scales. For this purpose, we introduce, for each $i$, $(X^{i,j})_{j=1,2,3}$, which are three independent simple random walks on $\mathcal{G}_{N_i}$, where $\mathcal{G}_n:=\mathcal{G}_{0.n}$, which was defined at the start of Section \ref{volsec-pre}. The law of this triple of walks, when started from $(x_1,x_2,x_3)$ will be denoted $P^{\mathcal{G}_{N_i}}_{x_1,x_2,x_3}$. Moreover, set
\[\bar{B}_n^\mathcal{G}:=\left\{x\in V(\mathcal{G}):\:\{x,y\}\in E(\mathcal{G})\mbox{ for some } y\in B_n\right\},\]
define $\bar{B}_n^{\mathcal{G}_{N_i}}$ similarly from $\mathcal{G}_{N_i}$, and suppose that the sets $\bar{B}_n^\mathcal{G}$ and $\bar{B}_n^{\mathcal{G}_{N_i}}$ are equipped with the edge sets induced from $\mathcal{G}$ and $\mathcal{G}_{N_i}$, respectively. Clearly, if
\[\mathcal{B}_i:=\left\{\bar{B}_{\eta N_i}^\mathcal{G}=\bar{B}_{\eta N_i}^{\mathcal{G}_{N_i}}\right\}\]
holds, then it is possible to couple the original random walks $(X^{j})_{j=1,2,3}$ and $(X^{i,j})_{j=1,2,3}$ (when the triples have the same starting points) in such a way that 
\begin{equation}\label{coupling}
X^{j,i}_n=X^j_n,\qquad \forall n\leq \Theta_{\eta N_i},\:j=1,2,3,
\end{equation}
and indeed the stopping time $\Theta_{\eta N_i}$ is the same for both versions of the random walks. (Since this will be the only case that matters in what follows, we will henceforth abuse notation and use the same symbol $\Theta_{\eta N_i}$ regardless of which triple of walks is under consideration, which should be clear from the context.) For convenience of notation, we write $\mathbb{P}$ for the annealed law of the processes $(X^j)_{j=1,2,3}$ and $(X^{i,j})_{j=1,2,3}$, $i\geq 1$, which we suppose are all started from 0 and built on the same probability space in such a way that \eqref{coupling} holds almost-surely on $\mathcal{B}_i$. For example, this means that 
\[\mathbb{P}((X^{i,j})_{j=1,2,3}\in\cdot):=\int P^{\mathcal{G}_{N_i}}_{0,0,0}\left(\cdot \right)P(dS).\]

Now, if we define $H_i'$ by setting
\[H_i'\coloneqq \sum_{n=\Theta_{\eta N_{i-1}}+1}^{(\Theta_{\eta N_{i-1}}+c_2N_i^2\Psi(N_i))\wedge \Theta_{\eta N_{i}}}\mathbbm{1}_{\{X^{i,1}_n=X^{i,2}_n=X^{i,3}_n\in A_{N_i}\}},\]
with the stopping times defined in terms of $(X^{i,j})_{j=1,2,3}$, on $\mathcal{B}_i$, it holds that 
\begin{equation}\label{hcomp}
H_i'\leq H_i.
\end{equation}
We highlight that we only have an inequality here because for $H_i'$ we are counting only the collisions that occur on the good set $A_{N_i}$. Moreover, we importantly have that the random variable $H_i$ is measurable with respect to the sigma-algebra:
\[\mathcal{H}_i:=\sigma\left(\mathcal{G}_{N_{i'}},\: \left(X^{i',j}_n\right)_{j=1,2,3,\:n\leq \Theta_{\eta N_{i'}}}:\:i'\leq i\right).\]

With these preparations in place, we are now ready to establish the two key inputs to proving Proposition \ref{hbound}, which are given by the following two lemmas. We emphasize that the argument of Lemma \ref{hdashbound} is rather novel. Although it follows the pattern of the corresponding argument in \cite{CDA}, it involves two applications of the second moment method -- one to a certain conditional mean of $H_i'$, and one to $H_i'$ itself. The reason for this is that it means we are able to weaken the control we require on the good set $A_{N_i}$ to what is provided by Proposition \ref{prop-conditional-AE}. (In particular, it is enough for the infimum to be outside the relevant probability.)

\begin{lemma}\label{bilem}
There exists $C>0$ such that 
\begin{equation}\label{lem-condi-B_i}
    P(\mathcal{B}_i\mid\mathcal{F}_{i-1})\ge 1-C(\log N_i)^{-1},
\end{equation}
holds for sufficiently large $i$, where $\mathcal{F}_i$ is as defined in \eqref{Fidef}. In particular, the event $\mathcal{B}_i$ holds for all large $i$ $P$-almost-surely. 
\end{lemma}

\begin{proof}
    Note that $\mathcal{B}_i$ holds if 
\[    S[0,\eta N_i]\cap S[N_i,\infty)=\emptyset.\]
It follows from the strong Markov property that 
\begin{equation*}
    P(S[N_{i-1},\eta N_i]\cap S[N_i,\infty)\neq\emptyset\mid\mathcal{F}_{i-1})
        =P(S[0,\eta N_i-N_{i-1}]\cap S[N_i-N_{i-1}]\neq\emptyset),
\end{equation*}
which further leads to, by the Markov property and reversibility, 
\begin{align*}
    P(S[0,\eta N_i-N_{i-1}]&\cap S[N_i-N_{i-1},\infty)\neq\emptyset)    \\
        &=P(S^1[(1-\eta)N_i,N_i-N_{i-1}]\cap S^2[0,\infty)\neq\emptyset),
\end{align*}
where $S^1$ and $S^2$ are independent simple random walks and $P$ on the right-hand side is the joint law conditioned that $S^1(0)=S^2(0)=0$. 
By \cite[Theorem 4.3.6]{Lawb}, there exists a universal constant $C>0$ such that 
\[    P(S^1[(1-\eta)N_i,N_i-N_{i-1}]\cap S^2[0,\infty)\neq\emptyset)\le C(\log N_i)^{-1},\]
for sufficiently large $i$ (recall the definition of $N_i$). In conclusion, we obtain 
\begin{equation}\label{ineq-BNi-1st}
    P(S[N_{i-1},\eta N_i]\cap S[N_i,\infty)\neq\emptyset\mid\mathcal{F}_{i-1})\le C(\log N_i)^{-1},
\end{equation}

To handle the intersection probability of $S[0,N_{i-1}]$ and $S[N_i,\infty)$, note that 
\[    S[0,N_{i-1}]\subset B(S(N_{i-1}),N_{i-1})\]
is trivial. Let $R_i\ge 1$, which is fixed later. Now we have  
\begin{align}
    P(&S[0,N_{i-1}]\cap S[N_i,\infty)\neq\emptyset\mid\mathcal{F}_{i-1})    \notag    \\
    &\le P(B(S(N_{i-1}),N_{i-1})\cap S[N_i,\infty)\neq\emptyset\mid\mathcal{F}_{i-1})   \notag\\
    &= P(B(0,N_{i-1})\cap S[N_i-N_{i-1},\infty)\neq\emptyset)   \notag\\
    &\le P(B(0,N_{i-1})\cap S[N_i-N_{i-1},\infty)\neq\emptyset,S(N_i-N_{i-1})\not \in B(R_i))    \notag\\
    &\qquad +P(S(N_i-N_{i-1})\in B(R_i)) \notag\\
    &\le \sup_{x\in \mathbb{Z}^4\setminus B(R_i)}P^x(B(0,N_{i-1})\cap S[0,\infty)\neq\infty)+P(S(N_i-N_{i-1})\in B(R_i))
    \label{ineq-BNi-intersect},
\end{align}
where we applied the strong Markov property in the equality and the Markov property in the latter inequality, respectively. 
By \cite[Theorem 1.2.1]{Lawb}, the second term is bounded above as 
\[    P(S(N_i-N_{i-1})\in B(R_i))\le R_i^4(N_i-N_{i-1})^{-2}.\]
By \cite[Theorem 1.5.10]{Lawb}, the first term on the right-hand side of \eqref{ineq-BNi-intersect} is bounded above as the following: 
\begin{equation*}
    \sup_{x\in \mathbb{Z}^4\setminus B(R_i)}P^x(B(0,N_{i-1})\cap S[0,\infty)\neq\infty)\le CR_i^{-2}N_{i-1}^2\le R_i^{-2}N_i^{\frac{(1-2\varepsilon)^2}{4}}.
\end{equation*}
Thus, taking $R_i=(N_i-N_{i-1})^{\frac{5}{12}}$, we obtain 
\[ P(S[0,N_{i-1}]\cap S[N_i,\infty)\neq\emptyset\mid\mathcal{F}_{i-1})\le N_i^{-\frac{1}{3}}.\]
Combining this\ with \eqref{ineq-BNi-1st}, we obtain \eqref{lem-condi-B_i}. 

The second statement follows from a rather straightforward application of \cite[Theorem 4.3.6]{Lawb}. This theorem implies 
\[    P(\mathcal{B}_i^c)\le P(S[0,\eta N_i]\cap S[N_i,\infty)\neq\emptyset)\le C(\log N_i)^{-1}.\]
From \eqref{nidef} it follows 
\begin{align*}
    \sum_iP(\mathcal{B}_i^c)\le \sum_i C4^{-i}<\infty,
\end{align*}
and the Borel-Cantelli lemma yields the desired statement. 
\end{proof}

\begin{proposition}\label{prop-B-tilde-i}
Let 
\[    \tilde{\mathcal{B}}_i\coloneqq \mathcal{B}_i\cap{\tilde{\mathcal{E}}}_{N'_i}\cap \{0\in A_{N_i}\},\]
where
\begin{equation}\label{nidashdef}
    N_i':=\eta N_i/C,
\end{equation}
with $C$ being the constant of Proposition \ref{hklower}. Then for sufficiently large $i$, 
\[    \inf_{m\in[N_i^{1-\varepsilon},N_i]}P\left(\{|A_{N_i}\cap B_m|\ge p'm\}\cap \tilde{\mathcal{B}}_{N_i}\mathrel{}\middle|\mathrel{} \mathcal{F}_{i-1}\right)\ge p'/2,\quad\mbox{a.s.}\]
\end{proposition}
\begin{proof}
    The statement follows from Lemma \ref{lem-vol-low-1}, Proposition \ref{prop-conditional-AE} and Lemma \ref{bilem}. 
\end{proof}

\begin{lemma}\label{hdashbound}
There exists a deterministic constant $c$ such that, $\mathbb{P}$-almost-surely, it holds that 
    \[\mathbb{P}\left(H'_i\geq c\left(\log N_i\right)^{1/2}\:\vline\:\mathcal{H}_{i-1}\right)\geq c.\]
    \end{lemma}
\begin{proof}
Writing $\mathbb{E}$ for the expectation associated with $\mathbb{P}$, define
\[E_i:=\mathbb{E}\left(H_i'\:\vline\:\mathcal{H}_{i-1},\:S\right)\]
for the conditional expectation of $H_i'$ given the random walk path $S$, and the walks at the previous level. By definition, we then have that
\begin{eqnarray}
    E_i&=& \mathbb{E}\left(\sum_{n=\Theta_{\eta N_{i-1}}+1}^{(\Theta_{\eta N_{i-1}+1}+c_2N_i^2\Psi(N_i))\wedge \Theta_{\eta N_{i}}}\mathbbm{1}_{\{X^{i,1}_n=X^{i,2}_n=X^{i,3}_n\in A_{N_i}\}}\:\vline\:\mathcal{H}_{i-1},\:S\right)\nonumber\\
    &\leq & \sum_{n=1}^{c_2N_i^2\Psi(N_i)}\sup_{(x_1,x_2,x_3)\in B_{\eta N_{i-1}}}\sum_{y\in A_{N_i}}P^{\mathcal{G}_{N_i}}_{(x_1,x_2,x_3)}\left(X^{i,1}_n=X^{i,2}_n=X^{i,3}_n=y,\:n\leq \Theta_{\eta N_i}\right)\nonumber\\
    &\leq& C\sum_{n=1}^{c_2N_i^2\Psi(N_i)}\sup_{(x_1,x_2,x_3)\in B_{\eta N_{i-1}}}\sum_{y\in A_{N_i}\cap B_{\eta N_i}}p_n^{\mathcal{G}_{N_i},B_{\eta N_i}}(x_1,y)p_n^{\mathcal{G}_{N_i},B_{\eta N_i}}(x_2,y)p_n^{\mathcal{G}_{N_i},B_{\eta N_i}}(x_3,y)\nonumber\\
    &\leq & C\sum_{n=1}^{c_2N_i^2\Psi(N_i)}\sup_{x\in B_{\eta N_{i-1}}}\sup_{y\in A_{N_i}\cap B_{\eta N_i}}p_n^{\mathcal{G}_{N_i},B_{\eta N_i}}(x,y)^2,\nonumber
    \end{eqnarray}
where we write $p_n^{\mathcal{G}_{N_i},B_{\eta N_i}}$ for the transition density of $X^{i,j}$, killed on exiting $B_{\eta N_i}$. Clearly, on $\mathcal{B}_i$, for the relevant $x$ and $y$, we have that
\[p_n^{\mathcal{G}_{N_i},B_{\eta N_i}}(x,y)=p_n^{B_{\eta N_i}}(x,y)\leq p_n(x,y)\leq \sqrt{p_{2\lfloor n/2\rfloor}(x,x)p_{2\lceil n/2\rceil}(y,y)},\]
where the second expression is the corresponding quantity for $X$ killed on exiting $B_{\eta N_i}$, the third expression is the heat kernel for the unkilled version of $X$, and the final inequality is obtained by applying Cauchy-Schwarz. Applying the heat kernel bounds of Proposition \ref{hkupper}, it follows that, on $\tilde{\mathcal{B}}_i$,
\begin{equation}
E_i\leq C\left(\sum_{n=1}^{c_1\exp(3\log^{1/2}(N_i))}n^{-1}+\sum_{n=c_1\exp(3\log^{1/2}(N_i))}^{c_2N_i^2\Psi(N_i)}{\frac{\Psi(n)}{n}}\right)\leq C\log^{1/2}(N_i),\label{dede}
\end{equation}
(Note that we applied here the fact that
\[\eta N_{i-1}\leq c_0\left(c_1\exp(3\log^{1/2}(N_i))/\Psi\left(c_1\exp(3\log^{1/2}(N_i))\right)\right)^{1/2},\]
and so the bounds of the second part of Proposition \ref{hkupper} apply to all $y\in B_{\eta N_{i-1}}$ for all relevant times.) Hence
\[\mathbb{E}\left(E_i^2\mathbf{1}_{\tilde{\mathcal{B}}_i}\:\vline\:\mathcal{H}_{i-1}\right)\leq C\log^{1/2}(N_i)\mathbb{E}\left(E_i\mathbf{1}_{\tilde{\mathcal{B}}_i}\:\vline\:\mathcal{H}_{i-1}\right),\]
and so, by the Paley-Zygmund inequality,
\begin{equation}\label{pz}
   \mathbb{P}\left(E_i\mathbf{1}_{\tilde{\mathcal{B}}_i}\geq \frac12 \mathbb{E}(E_i\mathbf{1}_{\tilde{\mathcal{B}}_i}\:|\:\mathcal{H}_{i-1})\:\vline\:\mathcal{H}_{i-1}\right)\geq \frac{\mathbb{E}(E_i\mathbf{1}_{\tilde{\mathcal{B}}_i}\:\vline\:\mathcal{H}_{i-1})^2}{4\mathbb{E}(E_i^2\mathbf{1}_{\tilde{\mathcal{B}}_i}\:\vline\:\mathcal{H}_{i-1})}\geq \frac{\mathbb{E}(E_i\mathbf{1}_{\tilde{\mathcal{B}}_i}\:\vline\:\mathcal{H}_{i-1})}{C\log^{1/2}(N_i)}. 
\end{equation}
It moreover holds that
\begin{eqnarray*}
    \lefteqn{\mathbb{E}(E_i\mathbf{1}_{\tilde{\mathcal{B}}_i}\:\vline\:\mathcal{H}_{i-1})}\\
    &=& \sum_{n=1}^{c_2N_i^2\Psi(N_i)}\mathbb{P}\left(X^{i,1}_{\Theta_{\eta N_{i-1}}+n}=X^{i,2}_{\Theta_{\eta N_{i-1}}+n}=X^{i,3}_{\Theta_{\eta N_{i-1}}+n}\in A_{N_i},\:\Theta_{\eta N_{i-1}}+n\leq \Theta_{\eta N_i},\:\tilde{\mathcal{B}}_i\:\vline\:\mathcal{H}_{i-1}\right)\\
&\geq  & \sum_{n=(N_i')^{2-\varepsilon}}^{\min\{c_2,\tilde{c}_3\}(N_i')^2\Psi(N_i')}\mathbb{E}\left(\mathbf{1}_{\tilde{\mathcal{B}}_i}\inf_{(x_1,x_2,x_3)\in B_{\eta N_{i-1}}}\sum_{y\in A_{N_i}\cap B_{\tilde{c}_2\sqrt{n/\Psi(n)}}}\prod_{j=1}^3p_n^{\mathcal{G}_{N_i},B_{CN_i'}}(x_j,y)\:\vline\:\mathcal{H}_{i-1}\right),
\end{eqnarray*}
where $\tilde{c}_2$, $\tilde{c}_3$ are the constants of Proposition \ref{hklower} and we have applied the fact that $B_{\eta N_i}=B_{CN_i'}$ (which is a consequence of the choice of $N_i'$ at \eqref{nidashdef}). Since we have that
\[\eta N_{i-1}\leq \tilde{c}_0\sqrt {(N_i')^{2-\varepsilon}\Psi\left((N_i')^{2-\varepsilon}\right)},\]
we can consequently apply the heat kernel lower bound of Proposition \ref{hklower}  (noting that $p_n^{\mathcal{G}_{N_i},B_{CN_i'}}=p_n^{\mathcal{G}_{N_i},B_{\eta N_i}}=p_n^{B_{\eta N_i}}$ on $\tilde{\mathcal{B}}_i$) to deduce that
\begin{eqnarray*}
\lefteqn{ \mathbb{E}(E_i\mathbf{1}_{\tilde{\mathcal{B}}_i}\:\vline\:\mathcal{H}_{i-1})}\\&\geq & C\sum_{n=(N_i')^{2-\varepsilon}}^{\min\{c_2,\tilde{c}_3\}(N_i')^2\Psi(N_i')}\mathbb{E}\left(\mathbf{1}_{\tilde{\mathcal{B}}_i}\left|A_{N_i}\cap B_{\tilde{c}_2\sqrt{n/\Psi(n)}}\right|\frac{\Psi(n)^{3/2}}{n^{3/2}}\:\vline\:\mathcal{H}_{i-1}\right)\\
&\geq & C\sum_{n=(N_i')^{2-\varepsilon}}^{\min\{c_2,\tilde{c}_3\}(N_i')^2\Psi(N_i')}\mathbb{P}\left(\left|A_{N_i}\cap B_{\tilde{c}_2\sqrt{n/\Psi(n)}}\right|\geq c\sqrt{\frac{n}{\Psi(n)}},\:\tilde{\mathcal{B}}_i\:\vline\:\mathcal{H}_{i-1}\right)\frac{\Psi(n)}{n}\\
&\geq&  C\sum_{n=(N_i')^{2-\varepsilon}}^{\min\{c_2,\tilde{c}_3\}(N_i')^2\Psi(N_i')}\frac{\Psi(n)}{n}\\
&\geq &C\log^{1/2}(N_i),
\end{eqnarray*}
where to obtain the penultimate inequality, we have applied Proposition \ref{prop-B-tilde-i}.

Returning to \eqref{pz}, we consequently deduce that
\begin{equation}\label{claim1}
    \mathbb{P}\left(E_i\geq c\log^{1/2}(N_i),\:\tilde{\mathcal{B}}_i\:\vline\:\mathcal{H}_{i-1}\right)= \mathbb{P}\left(E_i\mathbf{1}_{\tilde{\mathcal{B}}_i}\geq c\log^{1/2}(N_i)\:\vline\:\mathcal{H}_{i-1}\right)\geq c',
\end{equation}
which establishes (a stronger version of) the claim for the conditional means of $H_i'$.

We next apply the Paley-Zymund inequality to $H_i'$ itself, conditioning on $\mathcal{H}_{i-1}$ and $S$, to give 
\[ \mathbb{P}\left(H_i'\geq \frac12 E_i\:\vline\:\mathcal{H}_{i-1},\:S\right)\geq \frac{E_i^2}{4\mathbb{E}((H_i')^2\:\vline\:\mathcal{H}_{i-1},\:S)}.\]
Towards estimating the second moment here, we introduce the notation
\[\mathcal{C}_n^i:=\left\{X^{i,1}_{n}=X^{i,2}_{n}=X^{i,3}_{n}\in A_{N_i},\:n\leq \Theta_{\eta N_i}\right\}.\]
We then have that, on $\mathcal{B}_i$,
\begin{eqnarray*}
  \lefteqn{\mathbb{E}((H_i')^2\:|\:\mathcal{H}_{i-1},\:S)}\\
  &=&E_i
+2\sum_{n=\Theta_{\eta N_{i-1}}}^{\Theta_{\eta N_{i-1}}+c_2N_i^2\Psi(N_i)}\sum_{m=n+1}^{{\Theta_{\eta N_{i-1}}+c_2N_i^2\Psi(N_i)}}\mathbb{P}\left(\mathcal{C}^i_n\cap\mathcal{C}^i_m\:\vline\:\mathcal{H}_{i-1},\:S\right)\\  
&\leq & 
E_i
+2\sum_{n=\Theta_{\eta N_{i-1}}}^{\Theta_{\eta N_{i-1}}+c_2N_i^2\Psi(N_i)}\mathbb{P}\left(\mathcal{C}^i_n\:\vline\:\mathcal{H}_{i-1},\:S\right)\sum_{m=1}^{c_2N_i^2\Psi(N_i)}\sup_{x\in A_{N_i}\cap B_{\eta N_i}}\sum_{y\in  A_{N_i}\cap B_{\eta N_i}} p_m^{\mathcal{G}_{N_i},B_{\eta N_i}}(x,y)^3\\
&\leq &E_i
+2\sum_{n=\Theta_{\eta N_{i-1}}}^{\Theta_{\eta N_{i-1}}+c_2N_i^2\Psi(N_i)}\mathbb{P}\left(\mathcal{C}^i_n\:\vline\:\mathcal{H}_{i-1},\:S\right)\sum_{m=1}^{c_2N_i^2\Psi(N_i)} \sup_{x\in A_{N_i}\cap B_{\eta N_i}}p_m(x,x)^2\\
&\leq &E_i\left(1+2C\log^{1/2}(N_i)\right),
\end{eqnarray*}
where to deduce the first inequality, we apply the Markov property at time $n$, and for the final one, we apply the heat kernel estimates of Proposition \ref{hkupper}, similarly to the estimate at \eqref{dede}. From this bound, it follows that 
\begin{equation}\label{claim2}
\mathbb{P}\left(H_i'\geq \frac12 E_i\:\vline\:\mathcal{H}_{i-1},\:S\right)\geq \frac{E_i}{C\log^{1/2}(N_i)}.
\end{equation}

Finally, putting together \eqref{claim1} and \eqref{claim2} yields
\begin{eqnarray*}
\mathbb{P}\left(H_i'\geq c\log^{1/2}(N_i)\:\vline\:\mathcal{H}_{i-1}\right)&\geq& \mathbb{P}\left(H_i'\geq \frac12 E_i\geq c\log^{1/2}(N_i)\:\vline\:\mathcal{H}_{i-1}\right)\\
&\geq& \mathbb{E}\left(\mathbb{P}\left(H_i'\geq \frac12 E_i\:\vline\:\mathcal{H}_{i-1},\:S\right)\mathbf{1}_{\{E_i\geq c\log^{1/2}(N_i)\}}\:\vline\:\mathcal{H}_{i-1}\right)\\
&\geq & c'\mathbb{P}\left(E_i\geq c\log^{1/2}(N_i)\:\vline\:\mathcal{H}_{i-1}\right)\\
&\geq &c'',
\end{eqnarray*}
as required.
\end{proof}

We complete the section by combining the two previous results to give the main conclusion.

\begin{proof}[Proof of Proposition \ref{hbound}]
From Lemma \ref{hdashbound}, we have that, $\mathbb{P}$-almost-surely,
\[\sum_{i=1}^\infty \mathbb{P}\left(H'_i\geq c\left(\log N_i\right)^{1/2}\:\vline\:\mathcal{H}_{i-1}\right)=\infty.\]
Hence, since the sigma-algebras $\mathcal{H}_i$, 
$i\geq 1$, are increasing, and $H_i'$ is $\mathcal{H}_i$-measurable, we obtain from the second Borel-Cantelli lemma that $H'_i\geq c\left(\log N_i\right)^{1/2}$ for infinitely many $i$, $\mathbb{P}$-almost-surely. Combining this with Lemma \ref{bilem} and \eqref{hcomp}, we obtain the desired conclusion.    
\end{proof}

\section{Collisions of two independent random walks}\label{2coll}

The aim of this section is to prove Theorem \ref{mr2}, which provides a scaling limit for the collisions of two continuous-time random walks on $\mathcal{G}$. To this end, we will appeal to the general convergence result of \cite[Theorem 9.13]{Noda}. As noted in the introduction, this is the reason for considering continuous-time random walks rather than discrete-time ones. As a key ingredient for this, we have the convergence of the underlying spaces from \cite{CS}. The remaining condition that we need to check is the volume lower bound of \cite[Assumption 9.12(iv)]{Noda}, which concerns the volumes of balls on scales smaller than the typical scaling needed to establish convergence of random walks; this is achieved in Proposition \ref{lem-noda's condition} below.

Let us start by introducing some relevant notation. Firstly, for $n\in \mathbb{N}$ let 
\[\mathcal{X}_n:=\left(V(\mathcal{G}),\frac{R_{\mathcal{G}}}{n\Psi(n)},\frac{\mu^V_\mathcal{G}}{c_Vn},0\right),\]
be the metric space $(V(\mathcal{G}),\frac{R_{\mathcal{G}}}{n\Psi(n)})$ equipped with the normalised counting measure (in particular, $c_V$ is a constant, as determined by Proposition \ref{xnconv} below, and $\mu^V_\mathcal{G}(\{x\})=1$ for $x\in V(\mathcal{G})$ is the invariant measure of the processes $\tilde{X}^i$, $i=1,2$, as defined in the introduction), and we distinguish the origin 0 as the root of the space. We also define
\[\mathcal{X}:=\left([0,\infty),d_E,\mathcal{L},0\right),\]
where $d_E$ is the Euclidean metric on $[0,\infty)$, $\mathcal{L}$ the one-dimensional Hausdorff measure on the same space, and again distinguishing the point 0 as the root of the space. (We abuse notation by using the same symbol for the origin of both $\mathbb{R}^4$ and $\mathbb{R}$.)

By a minor adaptation of \cite[Proposition 3.1]{CS}, we have the following. The result gives the convergence in probability of $\mathcal{X}_n$ to $\mathcal{X}$ in the Gromov-Hausdorff vague topology. Since the latter concept is by now a standard notion of convergence for metric-measure spaces, we do not give a precise definition here, but simply refer the reader to \cite{ALWgap, Khezeli, Nodatop} for background.

\begin{proposition}[{cf.\ \cite[Proposition 3.1]{CS}}]\label{xnconv} There exists a deterministic constant $c_V\in(0,\infty)$ such that 
\[\mathcal{X}_n\rightarrow\mathcal{X}\]
in probability with respect to the Gromov-Hausdorff vague topology.
\end{proposition}
\begin{proof}
We only comment on the modifications from \cite{CS}. Firstly, the measure considered in \cite[Proposition 3.1]{CS} was the `degree measure', which gives each vertex a weight proportional to its degree. Changing this to the counting measure considered here can be achieved by a line-by-line replacement of the relevant arguments in \cite{CS}, and so we omit the details. Secondly, the result of \cite[Proposition 3.1]{CS} concerned the compact spaces
\[\mathcal{X}_n^{(r)}:=\left(B_R\left(0,rn\Psi(n)\right),\frac{R_{\mathcal{G}}}{n\Psi(n)},\frac{\mu^V_\mathcal{G}}{c_Vn},0\right)\]
and
\[\mathcal{X}^{(r)}:=\left([0,r],d_E,\mathcal{L},0\right)\]
(together with additional objects), giving that, for each $r>0$,
\[\mathcal{X}^{(r)}_n\rightarrow\mathcal{X}^{(r)}\]
in probability with respect to the Gromov-Hausdorff vague topology. However, it is also easy to check that, if the latter topology is metrised with respect to the metric of \cite[(2.5)]{ALWgap}, then the distance between $\mathcal{X}$ and $\mathcal{X}^{(r)}$, and that between $\mathcal{X}_n$ and $\mathcal{X}^{(r)}_n$, is bounded above by $e^{-r}$. Thus we obtain the result.
\end{proof}

We next give the additional volume bound that we require.

\begin{proposition}\label{lem-noda's condition}
There exist universal constants $C\in(0,\infty)$ and $c\in(0,1)$ such that for any fixed $\varepsilon\in(0,\frac{1}{2})$, $r\ge 1$ and $n\ge 1$, 
\begin{equation}\label{lem-2colli}
    P\left(\inf_{x\in B_R(0,rn\Psi(n))}V_R(x,\eta n\Psi(n)) \ge c\eta n \mbox{ for all }\eta \in\left((\log n)^{-1-\varepsilon},1\right]\right) \ge 1-Cr(\log n)^{-\frac{1}{2}+\varepsilon},
\end{equation}
holds.
\end{proposition}

\begin{proof}
Note that it suffices to prove \eqref{lem-2colli} for sufficiently large $n$ since we can extend it to all $n$ by taking $C$ large enough. 
We define events $F$ and $G$ by 
\begin{gather*}
    F=\{\exists k\in[2rn-n(\log n)^{-6},2rn],S[0,k]\cap S[k+1,\infty)=\emptyset\},
    \\ G=\left\{R(0,S(2rn))\ge \frac{3}{2}rn\Psi(n)\right\}.
\end{gather*}
By \cite[Lemma 7.7.4]{Lawb}, \cite[Lemma 2.3.1]{exact} and \cite[Theorem 1.1]{DS-SW-4d}, we have
\begin{equation}\label{2colli-FG}
    P(F\cap G)\ge 1-C\frac{\log \log n}{(\log n)^{1/2}},
\end{equation}
where $C_\theta>0$ is universal.  
Suppose $F\cap G$ occurs and take $K\in[2rn-n(\log n)^{-6},2rn]$ that satisfies $S[0,K]\cap S[K+1,\infty)=\emptyset\}$. 
From the triangle inequality, 
\[    R_\mathcal{G}(0,S(K))\ge R_\mathcal{G}(0,S(2rn))-R_\mathcal{G}(S(K),S(2rn))\ge rn\Psi(n), \]
follows since $R_\mathcal{G}(S(n),S(K))\le n(\log n)^{-6}$. 
By the definition of $K$, we also have that for all $l\ge 2rn+1$, 
\[    R_\mathcal{G}(0,S(l))=R_\mathcal{G}(0,S(K))+R_\mathcal{G}(S(K),S(l))\ge rn\Psi(n)+1.\]
This implies that 
\begin{equation}\label{lem-2colli-inclusion}
    V_R(0,rn\Psi(n))\subset B_{2rn},
\end{equation}
holds on $F\cap G$. 

Next, we prove \eqref{lem-2colli} in the case where the infimum is taken over $B_n$ and $\eta$ is dyadic. Let $\delta=1+\varepsilon$ and let 
\[    M=\left\lceil \frac{\delta}{\log 2}\log\log n\right\rceil+1.\]
Note that by \cite[Theorem 1.1]{DS-SW-4d}, $\Psi(2^{-m}n)\sim\Psi(n)$ for $m\in[1,M]$. 
Let $r$ be as fixed in the proof of \cite[Proposition 2.1.2]{DS'}. 
For $m=0,1,\cdots, M$ and $i=1,2,\cdots, 2^{m+1}r$, we define $I_i^{(m)}$ to be the event on which the following statements hold. 
\begin{gather*}
    \max_{(i-1)2^{-m}n\le k\le l\le i2^{-m}n} R_{\mathcal{G}_{k,l}}(S(k),S(l))\le 2^{-m+1}\Psi(n),
    \\ |S[(i-1)2^{-m}n,i2^mn]|\ge c2^{-m}n,
\end{gather*}
where $c$ is that of \eqref{andef}. 
By (2.17) of \cite{exact} and the proof of \cite[Proposition 2.1.2]{DS'}, 
\[    P\left(I_1^{(m)}\right)\ge 1-C(\log n)^{-3/2},\]
holds for all $m\in[1,M]$. 
Moreover, by translation invariance, we also have 
\begin{equation*}
    P\left(I_i^{(m)}\right)\ge 1-C(\log n)^{-3/2},
\end{equation*}
for all $m\in[0,M]$ and $i\in\{1,2,\cdots 2^{m+1}r\}$. 
Note that we can take $C$ uniformly in $n$, $m$ and $i$. 

Now, let $\displaystyle \mathcal{I}=\bigcap_{m=0}^M\bigcap_{i=1}^{2^m}\, I_i^{(m)}$. By taking the union bound, we have 
\begin{align}
    P(\mathcal{I}^c)&\le \sum_{m=0}^M C2^{m+1}r(\log n)^{-3/2}   \notag
    \\ &\le C2^{M+2}r(\log n)^{-3/2}\le Cr(\log n)^{-\frac{1}{2}+\varepsilon},\label{2colli-I}
\end{align}
where $C$ depends only on $\theta$. 
On the event $I_i^{(m)}$, it follows from the monotonicity of the effective resistance in subgraphs of $\mathcal{G}$ that 
\begin{align*}
    R_\mathcal{G}(S(k),S(l))\le R_{\mathcal{G}_{k,l}}(S(k),S(l))\le &\max_{(i-1)2^{-m}n\le k'\le l'\le i2^{-m}n} R_{\mathcal{G}_{k',l'}}(S(k'),S(l'))
    \\& \le 2^{-m+1}\Psi(n).
\end{align*}
for $(i-1)2^{-m}n\le k\le l\le i2^{-m}n$. This implies that 
\[    B_R(x,2^{-m+1}n\Psi(n))\supset S[(i-1)2^{-m}n,i2^{-m}n],\]
and, in particular, 
\[    V_R(x,2^{-m+1}n\Psi(n))\ge r2^{-m}n,\]
holds for any $x\in \mathcal{G}_{(i-1)2^{-m}n,i2^{-m}n}$. 
By monotonicity, we derive that on $\mathcal{I}$, 
\begin{equation}\label{lem-2colli-concl}    
    V_R(x,\eta n\Psi(n)) \ge \frac{c}{2}\eta n,
\end{equation}
holds for all $x\in B_{2rn}$ and $\eta\in((\log n)^{-1-\varepsilon},1]$. 

Finally, by \eqref{lem-2colli-inclusion}, \eqref{lem-2colli-concl} holds for all $x\in B_R(0,rn\Psi(n))$ and $\eta\in((\log n)^{-1-\varepsilon},1]$ on $F\cap G\cap \mathcal{I}$. 
Thus, the statement of the proposition follows from \eqref{2colli-FG} and \eqref{2colli-I}. 
\end{proof}

We are nearly ready to complete the proof of the main result. Before this, we introduce the limiting collision time process. In particular, the stochastic process naturally associated with $\mathcal{X}$, when the latter is viewed as a measured-resistance metric space (see \cite{Kigquasi} for background), is (up to a time-change by the deterministic constant 2) Brownian motion on $[0,\infty)$, started from 0 and reflected at the origin. Write $B^i=(B^i_t)_{t\geq 0}$, $i=1,2$, for two independent copies of this process. Moreover, let $(C_t^B)_{t\geq 0}$ be the positive continuous additive functional associated with the product process $(B^1,B^2)$, whose Revuz measure is given by the one-dimensional Hausdorff measure on $\{(x,x):\:x\in[0,\infty)\}$; this is the collision time process that appears in the statement of Theorem \ref{mr2} (see \cite[Section 7]{Noda} for details). It is a non-decreasing continuous-time stochastic process that increases only on the set of times $t$ where $B_t^1=B_t^2$, and we conclude the article by establishing that it is the scaling limit of the process introduced at \eqref{ctdef}.

\begin{proof}[Proof of Theorem \ref{mr2}]
Given that $(\tilde{X}^i_{c_Vtn^2\Psi(n)})$, $i=1,2$, are independent copies of the process naturally associated with $\mathcal{X}_n$, the result will follow from \cite[Theorem 9.13]{Noda} if we can check that \cite[Assumption 9.12]{Noda} holds in our setting. (We also discuss the appropriate reparameterisation below.) By Proposition \ref{xnconv}, we have that \cite[Assumption 9.12(i,ii)]{Noda} holds with $b_n=c_Vn$. Moreover, \cite[Assumption 9.12(iii)]{Noda} is given by \cite[Proposition 3.2]{CS}. And, \cite[Assumption 9.12(iv)]{Noda} with $\alpha_{r,\varepsilon}=\beta_{r,\varepsilon}=1$ is a direct consequence of Proposition \ref{lem-noda's condition}. Thus condition \cite[Assumption 9.12]{Noda} holds, and it follows that we have 
\[\left(\frac{1}{n\Psi(n)}C_{c_Vtn^2\Psi(n)}\right)_{t\geq 0}\buildrel{d}\over\rightarrow \left(C_t^B\right)_{t\geq 0}\]
in $D([0,\infty),\mathbb{R})$. The result is then obtained by a simple reparameterisation, using that $\Psi(n)\sim c_{\mathrm{res}}(\log n)^{-1/2}$, as noted at \eqref{cite-dssw-mainthm}. 
\end{proof}

\section*{Acknowledgments}

The authors would like to thank Umberto de Ambroggio for his contributions to the early part of the discussions that led to this article. The research was supported by JSPS Grant-in-Aid for Scientific Research (C)	24K06758, and the Research Institute for Mathematical Sciences, an International Joint Usage/Research Center located in Kyoto University. DS was supported by JSPS Grant-in-Aid for Scientific Research (C) 22K03336, JSPS Grant-in-Aid for Scientific Research (B) 22H01128
and 21H00989. SW was suppoeted by JSPS Grant-in-Aid for Research Activity Start-up 25K23335. 
    
\bibliographystyle{plain} 
\bibliography{4d}

@article {BCK,
    AUTHOR = {Barlow, M. T. and Coulhon, T. and Kumagai, T.},
     TITLE = {Characterization of sub-{G}aussian heat kernel estimates on
              strongly recurrent graphs},
   JOURNAL = {Comm. Pure Appl. Math.},
  FJOURNAL = {Communications on Pure and Applied Mathematics},
    VOLUME = {58},
      YEAR = {2005},
    NUMBER = {12},
     PAGES = {1642--1677},
      ISSN = {0010-3640,1097-0312},
   MRCLASS = {60J45 (60C05 60G50 60J10)},
  MRNUMBER = {2177164},
MRREVIEWER = {Endre\ Cs\'aki},
       DOI = {10.1002/cpa.20091},
       URL = {https://doi.org/10.1002/cpa.20091},
}

@article {BPS,
    AUTHOR = {Barlow, M. T. and Peres, Y. and Sousi, P.},
     TITLE = {Collisions of random walks},
   JOURNAL = {Ann. Inst. Henri Poincar\'e{} Probab. Stat.},
  FJOURNAL = {Annales de l'Institut Henri Poincar\'e{} Probabilit\'es et
              Statistiques},
    VOLUME = {48},
      YEAR = {2012},
    NUMBER = {4},
     PAGES = {922--946},
      ISSN = {0246-0203,1778-7017},
   MRCLASS = {60J10 (05C81 60J35 60J80)},
  MRNUMBER = {3052399},
MRREVIEWER = {Serguei\ Popov},
       DOI = {10.1214/12-AIHP481},
       URL = {https://doi.org/10.1214/12-AIHP481},
}

@article {Chenchen,
    AUTHOR = {Chen, X. and Chen, D.},
     TITLE = {Some sufficient conditions for infinite collisions of simple
              random walks on a wedge comb},
   JOURNAL = {Electron. J. Probab.},
  FJOURNAL = {Electronic Journal of Probability},
    VOLUME = {16},
      YEAR = {2011},
     PAGES = {no. 49, 1341--1355},
      ISSN = {1083-6489},
   MRCLASS = {60J10 (60G50 60K37)},
  MRNUMBER = {2827462},
MRREVIEWER = {Jonathon\ R.\ Peterson},
       DOI = {10.1214/EJP.v16-907},
       URL = {https://doi.org/10.1214/EJP.v16-907},
}

@article {CHllt,
    AUTHOR = {Croydon, D. A. and Hambly, B. M.},
     TITLE = {Local limit theorems for sequences of simple random walks on
              graphs},
   JOURNAL = {Potential Anal.},
  FJOURNAL = {Potential Analysis. An International Journal Devoted to the
              Interactions between Potential Theory, Probability Theory,
              Geometry and Functional Analysis},
    VOLUME = {29},
      YEAR = {2008},
    NUMBER = {4},
     PAGES = {351--389},
      ISSN = {0926-2601,1572-929X},
   MRCLASS = {60J35 (28A80 60G50)},
  MRNUMBER = {2453564},
MRREVIEWER = {Wilfried\ Hazod},
       DOI = {10.1007/s11118-008-9101-9},
       URL = {https://doi.org/10.1007/s11118-008-9101-9},
}

@article {CS,
    AUTHOR = {Croydon, D. A. and Shiraishi, D.},
     TITLE = {Scaling limit for random walk on the range of random walk in
              four dimensions},
   JOURNAL = {Ann. Inst. Henri Poincar\'e{} Probab. Stat.},
  FJOURNAL = {Annales de l'Institut Henri Poincar\'e{} Probabilit\'es et
              Statistiques},
    VOLUME = {59},
      YEAR = {2023},
    NUMBER = {1},
     PAGES = {166--184},
      ISSN = {0246-0203,1778-7017},
   MRCLASS = {60K37 (05C81 60K35 82B41)},
  MRNUMBER = {4533724},
       DOI = {10.1214/22-aihp1243},
       URL = {https://doi.org/10.1214/22-aihp1243},
}

@book {Kumsf,
    AUTHOR = {Kumagai, T.},
     TITLE = {Random walks on disordered media and their scaling limits},
    SERIES = {Lecture Notes in Mathematics},
    VOLUME = {2101},
      NOTE = {Lecture notes from the 40th Probability Summer School held in
              Saint-Flour, 2010,
              \'Ecole d'\'Et\'e{} de Probabilit\'es de Saint-Flour.
              [Saint-Flour Probability Summer School]},
 PUBLISHER = {Springer, Cham},
      YEAR = {2014},
     PAGES = {x+147},
      ISBN = {978-3-319-03151-4; 978-3-319-03152-1},
   MRCLASS = {60G50 (05C81 31C20 35K08 60K35 82B41)},
  MRNUMBER = {3156983},
MRREVIEWER = {Francis\ Comets},
       DOI = {10.1007/978-3-319-03152-1},
       URL = {https://doi.org/10.1007/978-3-319-03152-1},
}

@book {Lawb,
    AUTHOR = {Lawler, G. F.},
     TITLE = {Intersections of random walks},
    SERIES = {Modern Birkh\"auser Classics},
      NOTE = {Reprint of the 1996 edition},
 PUBLISHER = {Birkh\"auser/Springer, New York},
      YEAR = {2013},
     PAGES = {iv+223},
      ISBN = {978-1-4614-5971-2; 978-1-4614-5972-9},
   MRCLASS = {60G50 (60G17 60K99)},
  MRNUMBER = {2985195},
       DOI = {10.1007/978-1-4614-5972-9},
       URL = {https://doi.org/10.1007/978-1-4614-5972-9},
}

@book {LPW,
    AUTHOR = {Levin, D. A. and Peres, Y. and Wilmer, E. L.},
     TITLE = {Markov chains and mixing times},
      NOTE = {With a chapter by James G. Propp and David B. Wilson},
 PUBLISHER = {American Mathematical Society, Providence, RI},
      YEAR = {2009},
     PAGES = {xviii+371},
      ISBN = {978-0-8218-4739-8},
   MRCLASS = {60J10 (60-01 60J05 60K35 60K37 68U20 68W20)},
  MRNUMBER = {2466937},
MRREVIEWER = {Olle\ H\"aggstr\"om},
       DOI = {10.1090/mbk/058},
       URL = {https://doi.org/10.1090/mbk/058},
}

@book {LP,
    AUTHOR = {Lyons, R. and Peres, Y.},
     TITLE = {Probability on trees and networks},
    SERIES = {Cambridge Series in Statistical and Probabilistic Mathematics},
    VOLUME = {42},
 PUBLISHER = {Cambridge University Press, New York},
      YEAR = {2016},
     PAGES = {xv+699},
      ISBN = {978-1-107-16015-6},
   MRCLASS = {60C05 (05C05 05C81 28A80 60J10 60J80 60K35 82B41)},
  MRNUMBER = {3616205},
MRREVIEWER = {Laurent\ Miclo},
       DOI = {10.1017/9781316672815},
       URL = {https://doi.org/10.1017/9781316672815},
}

@article {KrishPeres,
    AUTHOR = {Krishnapur, M. and Peres, Y.},
     TITLE = {Recurrent graphs where two independent random walks collide
              finitely often},
   JOURNAL = {Electron. Comm. Probab.},
  FJOURNAL = {Electronic Communications in Probability},
    VOLUME = {9},
      YEAR = {2004},
     PAGES = {72--81},
      ISSN = {1083-589X},
   MRCLASS = {60B99 (60G50 60J10)},
  MRNUMBER = {2081461},
MRREVIEWER = {Kyle\ Siegrist},
       DOI = {10.1214/ECP.v9-1111},
       URL = {https://doi-org.kyoto-u.idm.oclc.org/10.1214/ECP.v9-1111},
}

@article {exact,
    AUTHOR = {Shiraishi, D.},
     TITLE = {Exact value of the resistance exponent for four dimensional
              random walk trace},
   JOURNAL = {Probab. Theory Related Fields},
  FJOURNAL = {Probability Theory and Related Fields},
    VOLUME = {153},
      YEAR = {2012},
    NUMBER = {1-2},
     PAGES = {191--232},
      ISSN = {0178-8051,1432-2064},
   MRCLASS = {82B41 (60G50)},
  MRNUMBER = {2925573},
MRREVIEWER = {Akira\ Sakai},
       DOI = {10.1007/s00440-011-0343-x},
       URL = {https://doi.org/10.1007/s00440-011-0343-x},
}

@article{exactcorr,
    author = {Croydon, D. A. and Shiraishi, D.},
    title = {Correction to: Exact value of the resistance exponent for four dimensional random walk trace},
    journal = {Probab. Theory Related Fields},
     VOLUME = {185},
      YEAR = {2023},
    NUMBER = {1-2},
     PAGES = {699--704},
}

@article {HP,
    AUTHOR = {Hutchcroft, T. and Peres, Y.},
     TITLE = {Collisions of random walks in reversible random graphs},
   JOURNAL = {Electron. Commun. Probab.},
  FJOURNAL = {Electronic Communications in Probability},
    VOLUME = {20},
      YEAR = {2015},
     PAGES = {no. 63, 6},
      ISSN = {1083-589X},
   MRCLASS = {60J10 (05C80 05C81 60K37)},
  MRNUMBER = {3399814},
MRREVIEWER = {Rongfeng\ Sun},
       DOI = {10.1214/ECP.v20-4330},
       URL = {https://doi-org.kyoto-u.idm.oclc.org/10.1214/ECP.v20-4330},
}

@article {CDA,
    AUTHOR = {Croydon, D. A. and De Ambroggio, U.},
     TITLE = {Triple collisions on a comb graph},
   JOURNAL = {Electron. J. Probab.},
  FJOURNAL = {Electronic Journal of Probability},
    VOLUME = {30},
      YEAR = {2025},
     PAGES = {Paper No. 130, 22},
      ISSN = {1083-6489},
   MRCLASS = {60J10 (05C81 60G50 60J35)},
  MRNUMBER = {4955057},
MRREVIEWER = {Abdessatar\ Souissi},
       DOI = {10.1214/25-ejp1392},
       URL = {https://doi-org.kyoto-u.idm.oclc.org/10.1214/25-ejp1392},
}

@unpublished{Noda,
    author = {Noda, R.},
    title = {Convergence of space-time occupation measures of stochastic processes and its application to collisions},
    note = {preprint available at arXiv:2510.19936}}

@article {CroyRW,
    AUTHOR = {Croydon, D. A.},
     TITLE = {Random walk on the range of random walk},
   JOURNAL = {J. Stat. Phys.},
  FJOURNAL = {Journal of Statistical Physics},
    VOLUME = {136},
      YEAR = {2009},
    NUMBER = {2},
     PAGES = {349--372},
      ISSN = {0022-4715,1572-9613},
   MRCLASS = {60G50 (60K35 60K37 82B41)},
  MRNUMBER = {2525250},
MRREVIEWER = {Neal\ Madras},
       DOI = {10.1007/s10955-009-9785-2},
       URL = {https://doi-org.kyoto-u.idm.oclc.org/10.1007/s10955-009-9785-2},
}

@article{DS-SW-4d,
      title={Graph distance and effective resistance of the four-dimensional random walk trace}, 
      author={Shiraishi, D. and Watanabe, S.},
      year={2026},
      journal = {arXiv preprint},
        volume = {arXiv:2602.17076},
      eprint={math.PR/2602.17076},
      pubstate      = {\bibstring{preprint}},
      archivePrefix={arXiv},
      primaryClass={math.PR},
      url={https://arxiv.org/abs/2602.17076}, 
}

@article {ALWgap,
    AUTHOR = {Athreya, S. and L\"ohr, W. and Winter, A.},
     TITLE = {The gap between {G}romov-vague and {G}romov-{H}ausdorff-vague
              topology},
   JOURNAL = {Stochastic Process. Appl.},
  FJOURNAL = {Stochastic Processes and their Applications},
    VOLUME = {126},
      YEAR = {2016},
    NUMBER = {9},
     PAGES = {2527--2553},
      ISSN = {0304-4149,1879-209X},
   MRCLASS = {60B05 (05C80 05C81 60B10 60B99 60J80)},
  MRNUMBER = {3522292},
MRREVIEWER = {David\ A.\ Croydon},
       DOI = {10.1016/j.spa.2016.02.009},
       URL = {https://doi-org.kyoto-u.idm.oclc.org/10.1016/j.spa.2016.02.009},
}

@article {Khezeli,
    AUTHOR = {Khezeli, A.},
     TITLE = {A unified framework for generalizing the {G}romov-{H}ausdorff
              metric},
   JOURNAL = {Probab. Surv.},
  FJOURNAL = {Probability Surveys},
    VOLUME = {20},
      YEAR = {2023},
     PAGES = {837--896},
      ISSN = {1549-5787},
   MRCLASS = {60B05 (51F30)},
  MRNUMBER = {4671147},
       DOI = {10.1214/20-ps340},
       URL = {https://doi-org.kyoto-u.idm.oclc.org/10.1214/20-ps340},
}

@unpublished{Nodatop,
    author = {Noda, R.},
    title = {Metrization of {G}romov-{H}ausdorff-type topologies on boundedly-compact metric spaces},
    note = {preprint available at arXiv:2404.19681}}

@article {Kigquasi,
    AUTHOR = {Kigami, J.},
     TITLE = {Resistance forms, quasisymmetric maps and heat kernel
              estimates},
   JOURNAL = {Mem. Amer. Math. Soc.},
  FJOURNAL = {Memoirs of the American Mathematical Society},
    VOLUME = {216},
      YEAR = {2012},
    NUMBER = {1015},
     PAGES = {vi+132},
      ISSN = {0065-9266,1947-6221},
      ISBN = {978-0-8218-5299-6},
   MRCLASS = {30L10 (28A80 31C25 35K05 35K08 60J45)},
  MRNUMBER = {2919892},
MRREVIEWER = {Leonid\ V.\ Kovalev},
       DOI = {10.1090/S0065-9266-2011-00632-5},
       URL = {https://doi-org.kyoto-u.idm.oclc.org/10.1090/S0065-9266-2011-00632-5},
}

@book {Kall,
    AUTHOR = {Kallenberg, O.},
     TITLE = {Foundations of modern probability},
    SERIES = {Probability and its Applications (New York)},
   EDITION = {Second},
 PUBLISHER = {Springer-Verlag, New York},
      YEAR = {2002},
     PAGES = {xx+638},
      ISBN = {0-387-95313-2},
   MRCLASS = {60-01},
  MRNUMBER = {1876169},
MRREVIEWER = {Klaus\ D.\ Schmidt},
       DOI = {10.1007/978-1-4757-4015-8},
       URL = {https://doi-org.kyoto-u.idm.oclc.org/10.1007/978-1-4757-4015-8},
}

@article {Nguyen,
    AUTHOR = {Nguyen, D.-T.},
     TITLE = {Scaling limit of the collision measures of multiple random
              walks},
   JOURNAL = {ALEA Lat. Am. J. Probab. Math. Stat.},
  FJOURNAL = {ALEA. Latin American Journal of Probability and Mathematical
              Statistics},
    VOLUME = {20},
      YEAR = {2023},
    NUMBER = {2},
     PAGES = {1385--1410},
      ISSN = {1980-0436},
   MRCLASS = {60G50 (60F05 60G57 82C05)},
  MRNUMBER = {4683379},
       DOI = {10.30757/alea.v20-52},
       URL = {https://doi-org.kyoto-u.idm.oclc.org/10.30757/alea.v20-52},
}

@article {DStams,
    AUTHOR = {Shiraishi, D.},
     TITLE = {Random walk on non-intersecting two-sided random walk is subdiffusive
in low dimensions},
   JOURNAL = {Transactions of the American Mathematical Society},
  FJOURNAL = {Transactions of American Mathematical Society},
    VOLUME = {370},
      YEAR = {2018},
    NUMBER = {7},
     PAGES = {4525--4558},
      ISSN = {},
   MRCLASS = {},
  MRNUMBER = {},
       DOI = {},
       URL = {},
}

\end{document}